\documentclass[11pt]{article}

\usepackage[left=2.5cm, right=2.5cm, top=2.5cm, bottom=2.5cm]{geometry}

\usepackage{amsfonts, amsmath, hyperref, latexsym, amssymb, amsthm, url,tikz, xcolor, bm, ifthen}
\usepackage{setspace}
\usepackage[affil-it]{authblk}
\usepackage{url}

\usetikzlibrary{shapes.geometric}
\usetikzlibrary{positioning}
\newtheorem{theorem}{Theorem}[section]
\newtheorem{lemma}[theorem]{Lemma}
\newtheorem{corollary}[theorem]{Corollary}

\newtheorem{proposition} [theorem]{Proposition}

\newtheorem{remark}[theorem]{Remark}

\title{Broadcasting induced colourings of random recursive trees and preferential attachment trees}
\author[1]{Colin Desmarais\thanks{Supported by the Swedish Research Council, the Knut and Alice Wallenberg Foundation, and the Swedish Foundation's starting grant from the Ragnar S\"oderberg Foundation.}}
\author[1]{Cecilia Holmgren$^*$} 
\author[1,2]{Stephan Wagner\thanks{Supported by the Knut and Alice Wallenberg Foundation.}}
\affil[1]{\small Department of Mathematics, Uppsala University, Uppsala, Sweden\\ \url{{colin.desmarais, cecilia.holmgren, stephan.wagner}@math.uu.se}}
\affil[2]{\small Department of Mathematical Sciences, Stellenbosch University, Stellenbosch, South Africa}

\begin{document}

\maketitle

\begin{abstract}
In this work we consider random two-colourings of random linear preferential attachment trees, which includes random recursive trees, random plane-oriented recursive trees, random binary search trees, and a class of random $d$-ary trees. The random colouring is defined by assigning the root of the tree the colour red or blue with equal probability, and all other vertices are assigned the colour of their parent with probability $p$ and the other colour otherwise. These colourings have been previously studied in other contexts, including Ising models and broadcasting, and can be considered as generalizations of bond percolation. With the help of P\'olya urns, we prove limiting distributions, after proper rescalings, for the number of vertices of each colour, the number of monochromatic subtrees of each colour, as well as the number of leaves and fringe subtrees with two-colourings. Using methods from analytic combinatorics, we also provide precise descriptions of the limiting distribution after proper rescaling of the size of the root cluster; the largest monochromatic subtree containing the root. The description of the limiting distributions
extends previous work on bond percolation in random preferential attachment trees. \\

\noindent {\bf Keywords:} Random trees, preferential attachment, broadcasting, percolation, P\'olya urns, generating functions.
\end{abstract}

\section{Introduction}
\normalsize

For a rooted tree $T = (V,E)$ with root $\rho$ and $p \in [0,1]$, we define a {\em broadcasting induced colouring} $\sigma_{T,p}$ of $T$ to be a random 2-colouring $\sigma_{T,p} : V(T) \rightarrow \{ red, blue \}$ of the vertices of $T$ such that 
\pagebreak
\begin{itemize}
\item[(i)] $\mathbb{P}(\sigma_{T,p}(\rho) = \text{red}) = \mathbb{P}(\sigma_{T,p}(\rho)= \text{blue}) = \frac{1}{2}$ and
\item[(ii)] for all vertices $v$ with parent $u$, $\mathbb{P}(\sigma_{T,p}(v)=\sigma_{T,p}(u)) = p$ and $\mathbb{P}(\sigma_{T,p}(v) \neq \sigma_{T,p}(u)) = 1-p$.
\end{itemize}

The colouring $\sigma_{T,p}$ is induced from a broadcast process, in which the root of $T$ is assigned a bit 0 or 1 uniformly at random, and this bit is propagated along the tree in the following way: any vertex in the tree takes the same bit as its parent with probability $p$ and the other bit with probability $1-p$. By assigning a vertex the colour red if its bit value is 0 and blue if its bit value is 1, we recover the colouring $\sigma_{T,p}$. This broadcast process was described by Evans, Kenyon, Peres, and Schulman \cite{EKPS:00}, where they outline a correspondence of this process to the Ising model (see \cite[Section 2.2]{EKPS:00}). The reconstruction problem is then to reconstruct the bit value of the root $\rho$ from the bit values of some subset of vertices in $T$ after broadcasting. This problem has long been studied, see for example the survey \cite{MOSS:04} of early works. Applications of the reconstruction problem in trees include its connection to stochastic block models \cite{MONS:16, GULM:18}, a random graph model with applications in machine learning. Of particular interest to this work, Addario-Berry, Devroye, Lugosi, and Velona studied the reconstruction problem in random recursive trees and preferential attachment trees \cite{ADLV:20}. 

For a real number $\alpha$, a random (linear) preferential attachment tree $\mathcal{T}_{\alpha,n}$ is grown recursively in the following manner. The tree $\mathcal{T}_{\alpha,1}$ consists of a single vertex $\rho$, the root of all trees that follow. The tree $\mathcal{T}_{\alpha,n}$ is grown from $\mathcal{T}_{\alpha,n-1}$ by choosing a vertex $v$ at random and adding a child to $v$, where $v$ is chosen with probability 
\begin{equation}\label{eq:prefprob}
\frac{\alpha\deg^+(v) + 1 }{\sum_{u \in V(\mathcal{T}_{\alpha,n-1})} (\alpha \deg^+(u) + 1)}
\end{equation}
where $\deg^+(u)$ (called the {\em outdegree} of $u$) is the number of children of $u$. To avoid degenerate cases, we only allow $\alpha \in \{\ldots, -\frac{1}{4}, -\frac{1}{3}, -\frac{1}{2}\} \,\bigcup \,[0,\infty)$. If $\alpha = -1$, then only leaves can be chosen as the parent of a new vertex, resulting in $\mathcal{T}_{-1,n}$ simply being a path of length $n$. If $\alpha$ is a different negative number outside of $\{\ldots, -\frac{1}{4}, -\frac{1}{3}, -\frac{1}{2}\}$, then there may be vertices $v$ in $\mathcal{T}_{\alpha,n}$ for which \eqref{eq:prefprob} is negative. This problem is avoided when $\alpha = -\frac{1}{d}$, since \eqref{eq:prefprob} is positive when $\deg^+(u) < d$ and is zero when $\deg^+(u) = d$, resulting in a tree $\mathcal{T}_{-1/d, n}$ whose vertices all have outdegree less than or equal to $d$. 

The random tree $\mathcal{T}_{\alpha,n}$ has several names in the literature. When $\alpha = 0$, the vertex $v$ is chosen uniformly at random amongst all the vertices in the 
tree. This random tree is called a {\em random recursive tree}, and has been extensively studied 
for many years; since at least 1967 \cite{TAMY:67}. When $\alpha =1$, the tree $\mathcal{T}_n$ is called a 
{\em random plane-oriented recursive tree} which was introduced by Szyma\'nski \cite{SZYM:87}. The more general {\em linear preferential attachment tree} $\mathcal{T}_{\alpha,n}$ coincides with a special case of the preferential attachment model studied by Barab\'asi and Albert \cite{BAAL:99}, but has also been studied in several other contexts (see for example \cite{PITT:94,BRTS:01,HOJS:17}). 
When $\alpha = -\frac{1}{d}$ for a positive integer $d$, the tree $\mathcal{T}_{-1/d, n}$ is a model of random $d$-ary trees, and corresponds to a random binary search tree when $d=2$. The random trees $\mathcal{T}_{\alpha,n}$ also fall into the class of {\em increasing trees} (see \cite{BEFS:92,DRMO:09}), so named since if we label the vertices $1, \ldots,n$ by the time they appear in the tree, then the labels increase along all paths from the root.

For ease of notation, we may sometimes fix $\alpha$ and $p$, and let $\mathcal{T}_n$ denote the tree $\mathcal{T}_{\alpha, n}$ and let $\sigma_n$ denote the random broadcasting induced colouring $\sigma_{\mathcal{T}_n, p}$. For fixed $\alpha$ and $p$ we can consider a random sequence $((\mathcal{T}_n, \sigma_n))_{n=1}^\infty$ of preferential attachment trees with broadcasting induced colourings where $\mathcal{T}_n$ is grown from $\mathcal{T}_{n-1}$ in the manner outlined above, and where $\sigma_n$ restricted to the $n-1$ vertices of $\mathcal{T}_{n-1}$ is equal to $\sigma_{n-1}$ (and the colour of the newest vertex $v$ in $\mathcal{T}_n$ is randomly chosen such that with probability $p$ the colour of $v$ is the same as its parent). 

As an example of the growth process we describe, consider the trees with broadcasting induced colourings in Figure \ref{fig:example}. The tree $\mathcal{T}_8$ is grown from $\mathcal{T}_7$ by choosing $v$ according to the probability \eqref{eq:prefprob} and adding a child $u$ (notice that $\deg^+(v) = 0$). The probability that $u$ takes a different colour from $v$ is $1-p$, and so 

\[ \mathbb{P}((\mathcal{T}_8, \sigma_8) | (\mathcal{T}_7, \sigma_7)) = (1-p)\left(\frac{1}{6\alpha + 7}\right).\]

\begin{figure}[h!]
\centering
\begin{tikzpicture}[
redv/.style={circle, inner sep=0pt, draw=black, fill=red, thick, minimum size=8pt},
bluv/.style={rectangle, inner sep=0pt, draw=black, fill=cyan, thick, minimum size=8pt},
scale=1]

\draw (5,4) -- (4,3) -- (3,2) ;
\draw (5,4) -- (6,3) -- (7,2);
\draw (5,4) -- (5,3);
\draw (6,3) -- (5,2);
\node[redv] at (5,4){};

\node[redv] at (4,3){};
\node[redv] at (5,3){};
\node[below right] at (5,3) {$v$};
\node[redv] at (6,3){};

\node[bluv] at (3,2){};
\node[redv] at (5,2){};

\node[bluv] at (7,2){};

\node at (5, 1){$(\mathcal{T}_{7}, \sigma_7)$};

\node at (8.5, 3){$\longrightarrow$};

\draw (12,4) -- (11,3) -- (10,2) ;
\draw (12,4) -- (13,3) -- (14,2);
\draw (12,4) -- (12,3) -- (11,2);
\draw (13,3) -- (12,2);

\node[redv] at (12,4){};

\node[redv] at (11,3){};
\node[redv] at (12,3){};
\node[below right] at (12,3) {$v$};
\node[redv] at (13,3){};

\node[bluv] at (10,2){};
\node[bluv] at (11,2){};
\node[below right] at (11,2) {$u$};
\node[redv] at (12,2){};

\node[bluv] at (14,2){};

\node at (12, 1){$(\mathcal{T}_{8}, \sigma_8)$};

\end{tikzpicture}
\caption{A tree $\mathcal{T}_8$ with broadcasting induced colouring $\sigma_8$ grown from $(\mathcal{T}_7, \sigma_7)$. }
\label{fig:example}
\end{figure}
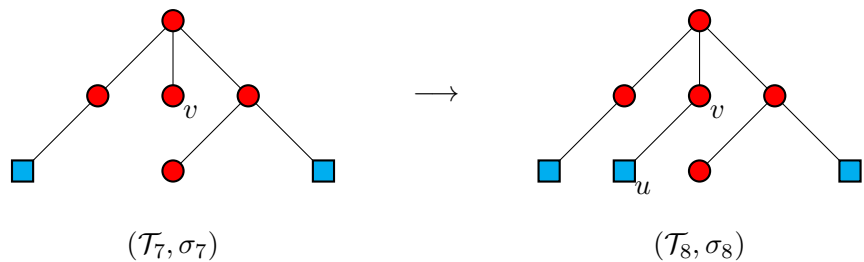

Our contribution in this work is to study asymptotic properties of the trees in the sequence $((\mathcal{T}_n, \sigma_n))_{n=1}^\infty$. These results are gathered in Section \ref{sec:results}, and come in two categories. In Section \ref{sec:resultsglobal} we list global properties of the trees $((\mathcal{T}_n, \sigma_n))_{n=1}^\infty$. These include limit laws (after appropriate rescaling) for the number of vertices of each colour, the number of clusters of each colour (maximal monochromatic subtrees of $\mathcal{T}_{n}$), the number of leaves of each colour, and the number of trees $T_1, \ldots, T_m$ with respective 2-colourings $\varsigma_1, \ldots, \varsigma_m$ appearing in the fringe. These results are proved using results on P\'olya urns from \cite{JANS:04}. The limiting distributions experience different phases. When $p < (3-\alpha)/4$, we observe normal limit laws after rescaling by $\sqrt{n}$. Normal limit laws are also observed when $p = (3-\alpha)/4$ but with a rescaling factor of $\sqrt{n \ln n}$, while convergence to a non-normal distribution is observed when $p > (3 - \alpha)/4$. 

In Section \ref{sec:resultsrootcluster} we study the size of the cluster $\mathcal{C}_n$ containing the root $\rho$. If we consider $\mathcal{T}_{n}$ with a random broadcasting induced colouring $\sigma_{n}$ and remove edges between two vertices if they do not have the same colour, we are left with a forest of trees corresponding to clusters after performing Bernoulli bond percolation with parameter $p$ on $\mathcal{T}_{n}$; Bernoulli bond percolation with parameter $p$ is a process by which each edge in a graph is kept with probability $p$ and removed with probability $1-p$, independently of every other edge. The size of $\mathcal{C}_n$ in the context of percolation, with a connection to memory-reinforced random walks, has previously been studied (see \cite{BAUR:20, BABE:15, KURS:16, BUSI:18}). In particular, Businger \cite{BUSI:18} has shown that for random recursive trees, $|\mathcal{C}_n|/n^p$ converges in distribution to a Mittag-Leffler distribution (see also \cite[Theorem 3.1]{BABE:15} and \cite{MOHL:15}). Baur \cite{BAUR:20} studied $|\mathcal{C}_n|$ in linear preferential attachment trees with $\alpha \geq 0$. He showed that $|\mathcal{C}_n|/n^{(p+\alpha)/(1+\alpha)}$ converges in distribution to some random variable $\mathcal{C}$, and provided the first two moments of this random variable. In this paper, we reprove these earlier results, and give a more precise description of this random variable $\mathcal{C}$ by providing a recursion to calculate the integer moments of $\mathcal{C}$. These results are proved using methods of analytic combinatorics. When $\alpha = 1$ we give a closed form for the integer moments of $\mathcal{C}$. We further extend these results by studying the size of $\mathcal{C}_n$ when $\alpha = -\frac{1}{d}$, where we observe different phases. When $p > \frac{1}{d}$, we observe a similar limiting distribution as when $\alpha > 0$, and find closed forms for the integer moments of this limiting distribution when $d=2$. When $p \leq \frac{1}{d}$, the size of the root cluster $\mathcal{C}_n$ is bounded almost surely as $n \to \infty$, and we describe the limiting distribution of $|\mathcal{C}_n|$ as the size of a Galton-Watson tree with binomial $\text{Bin}(d,p)$ offspring distribution.

\section{Main results}\label{sec:results}
In this section, we gather our main results. They are seperated in two categories: global properties, and the size of the root cluster. 

\subsection{Global properties}\label{sec:resultsglobal}

Throughout this section we define a random variable 
\begin{equation}\label{eq:rootcolour}
B = 
\begin{cases} 
1 & \text{if the root is red},\\
-1 & \text{if the root is blue}.
\end{cases}
\end{equation}
Then $B \sim (-1)^{\text{Be}(1/2)}$, where $\text{Be}$ is a Bernoulli random variable. 

We start with the number of vertices of each colour in a random preferential attachment tree $\mathcal{T}_n = \mathcal{T}_{\alpha,n}$ with a broadcasting induced colouring $\sigma_n = \sigma_{\mathcal{T}_n, p}$. 

\begin{theorem}\label{thm:vts}
Let $R_n$ and $B_n$ denote the number of red and blue vertices respectively in a preferential attachment tree $\mathcal{T}_{n}$ with broadcasting induced colouring $\sigma_{n}$. 
\begin{enumerate}
\item[(i)] The following strong law of large number holds
\begin{equation*}\label{eq:vtsa}
 \frac{1}{n}(R_n, B_n) \xrightarrow{a.s.} \left(\frac{1}{2}, \frac{1}{2}\right).
 \end{equation*}
\item[(ii)] If $p < (3- \alpha)/4$ or if $p=1/2$, then the following multivariate normal limit law holds
\begin{equation*}\label{vts<}
 \frac{(R_n, B_n) - n\left(\frac{1}{2}, \frac{1}{2}\right)}{\sqrt{n}} \xrightarrow{d} \mathcal{N}(\bm{0}, \Sigma_I) 
 \end{equation*}
where 
\begin{align*}
 \Sigma_I = c_{\alpha,p}
\left( \begin{array}{cc}
1 & -1 \\
-1 & 1 
\end{array}\right), 
&\hspace{10mm}c_{\alpha,p} = \begin{cases}
\frac{4\alpha p - \alpha - 1}{4(4p + \alpha - 3)} & p \neq \frac{1}{2}, \\
\frac{1}{4} & p = \frac{1}{2}.
\end{cases}
\end{align*}
\item[(iii)] If $p= (3-\alpha)/4$ and $p \neq \frac12$ (i.e., $\alpha \neq 1$), then the following multivariate normal limit law holds
\begin{equation*}\label{eq:vts=}
 \frac{(R_n, B_n) - n\left(\frac{1}{2}, \frac{1}{2}\right)}{\sqrt{n\ln{n}}} \xrightarrow{d} \mathcal{N}(\bm{0}, \Sigma_{II}) 
 \end{equation*}
where 
\[ \Sigma_{II} = \frac{(\alpha-1)^2}{4(1+\alpha)}
\left( \begin{array}{cc}
1 & -1 \\
-1 & 1 
\end{array}\right) .\]

\item[(iv)] If $p > (3-\alpha)/4$ and $p \neq \frac12$, then the following convergence in distribution holds
\begin{equation}\label{eq:vts>}
\frac{(R_n, B_n)- n\left(\frac{1}{2}, \frac{1}{2}\right)}{n^{(2p+\alpha - 1)/(1 + \alpha)}} \xrightarrow{d} \frac{B Z}{2(2p+\alpha - 1)}(2p-1,1-2p),
\end{equation}
where $Z$ is a random variable with 
\[ \mathbb{E}[Z] = \frac{\Gamma(1/(1 + \alpha))}{\Gamma((2p + \alpha)/(1 + \alpha))},\]
and 
\[\mathbb{E}[Z^2]  = \frac{\Gamma(1/(1 + \alpha))(1 + \alpha)(4p + \alpha-2)}{\Gamma((4p + 2\alpha -  1)/(1 + \alpha))(4p + \alpha - 3)}.\]
\end{enumerate}
\end{theorem}

\begin{remark}\label{rem:vts,p=1/2}
When $p=1/2$, every new vertex added to the tree is either red or blue with probability $1/2$, independent of everything that happened before. Therefore, $R_n$ is simply a sum of independent Bernoulli $\text{Be}(1/2)$ random variables, as is $B_n= n - R_n$. We see that in this case, the matrix $\Sigma_I$ in the convergence (ii) simplifies to 
\[ \Sigma_I = \frac{1}{4}\left(\begin{array}{cc} 1 & -1 \\ -1 & 1 \end{array}\right),\]
and we see that the random variables on the right hand side of the convergences in (iii) and (iv) degenererate to (0,0).
\end{remark}

We now turn to the number of clusters (maximal monochromatic subtrees). If we want to know the total number of clusters, we can first notice that whenever a child takes a different colour from its parent, a new cluster is formed. This is also the only way of forming a new cluster (in addition to the initial cluster containing the root). The probability that a newly added vertex does not take the colour of its parent is $1-p$, from which we can conclude that the total number of clusters at time $n$ is simply $1 + \text{Bin}(n-1, 1-p)$, where $\text{Bin}$ denotes a binomial random variable. 

\begin{theorem}\label{thm:clusters}
Let $R_n^c$ and $B_n^c$ denote the number of red and blue clusters respectively in a preferential attachment tree $\mathcal{T}_{n}$ with broadcasting induced colouring $\sigma_{n}$. 
\begin{enumerate}
\item[(i)] The following strong law of large numbers holds
\begin{equation*}\label{eq:clustersas}
 \frac{1}{n}(R_n^c, B_n^c) \xrightarrow{a.s.} \left(\frac{1-p}{2}, \frac{1-p}{2}\right).
 \end{equation*}

\item[(ii)]  If $p < (3- \alpha)/4$, then the following multivariate normal limit law holds
\begin{equation*}\label{eq:clusters<}
 \frac{(R_n^c, B_n^c) - n\left(\frac{1-p}{2}, \frac{1-p}{2}\right)}{\sqrt{n}} \xrightarrow{d} \mathcal{N}(\bm{0}, \Sigma^c_I),
 \end{equation*}
where 
\[ \Sigma^c_I = \frac{1-p}{4(3 - \alpha - 4p)}
\left( \begin{array}{cc}
(1-p)(\alpha + 4p + 1) & 3p - 4p^2 - \alpha p - \alpha - 1 \\
3p - 4p^2 - \alpha p - \alpha - 1 & (1-p)(\alpha + 4p + 1) 
\end{array} \right).\]

\item[(iii)] If $p = (3-\alpha)/4$, then the following multivariate normal limit law holds
\begin{equation*}\label{eq:clusters=}
 \frac{(R_n^c, B_n^c) - n\left(\frac{1-p}{2}, \frac{1-p}{2}\right)}{\sqrt{n\ln{n}}} \xrightarrow{d} \mathcal{N}(\bm{0}, \Sigma^c_{II}),
 \end{equation*}
where 
\[ \Sigma^c_{II} = \frac{\alpha+1}{16}
\left( \begin{array}{cc}
1 & -1 \\
-1 & 1 
\end{array}\right) .\]

\item[(iv)] If $p > (3-\alpha)/4$, then the following convergence in distribution holds
\begin{equation*}\label{eq:clusters>}
\frac{(R_n^c, B_n^c)- n\left(\frac{1-p}{2}, \frac{1-p}{2}\right)}{n^{(2p+\alpha - 1)/(1 + \alpha)}} \xrightarrow{d} \frac{B Z}{2(2p+\alpha -1)}(p-1,1-p),
\end{equation*}
where $Z$ is the same random variable as in \eqref{eq:vts>}, 
\end{enumerate}
\end{theorem}

We finish our summary of global properties with fringe subtrees. In a rooted tree $\mathcal{T}$ a fringe subtree $T$ consists of a vertex and all its descendents. The simplest example of a fringe subtree in $\mathcal{T}$ is a vertex with no descendents (a leaf of $\mathcal{T}$). Normal limit laws for the number of leaves in preferential attachment trees are already well known (see \cite{NAHE:82,MASM:92, JANS:05, HOJS:17}). 

In this simplest case, 
we offer covariance matrices for the limiting normal limit laws for 
the number of leaves of each colour in $\mathcal{T}_{n}$, though the distributions are already markedly more complicated than what we have described above.

\begin{theorem}\label{thm:leaves}
Let $R_n^l$ and $B_n^l$ denote the number of red and blue leaves respectively in a preferential attachment tree $\mathcal{T}_{n}$ with broadcasting induced colouring $\sigma_{n}$. 
\begin{enumerate}
\item[(i)] The following strong law of large numbers holds
\begin{equation*}\label{eq:leavesas}
 \frac{1}{n}(R_n^l, B_n^l) \xrightarrow{a.s.} \left(\frac{1+\alpha}{4+2\alpha}, \frac{1+\alpha}{4+2\alpha}\right).
 \end{equation*}
\item[(ii)] If $p < (3- \alpha)/4$ or if $p=1/2$, then the following multivariate normal limit law holds
\begin{equation*}\label{eq:leaves<}
 \frac{(R_n^l, B_n^l) - n\left(\frac{1+\alpha}{4+2\alpha}, \frac{1+\alpha}{4+2\alpha}\right)}{\sqrt{n} }\xrightarrow{d} \mathcal{N}(\bm{0} , \Sigma_I^l), 
 \end{equation*}
where 
\[ \Sigma_I^l = \frac{\alpha + 1}{4(2+\alpha)^2(3+\alpha)(2p-3)(4p+\alpha-3)}
\left(\begin{array}{cc}
\sigma^l_{1,1} & \sigma^l_{1,2} \\
\sigma^l_{2,1} & \sigma^l_{2,2}
\end{array}\right),\]
with 
\begin{align*}
\sigma^l_{1,1} &= \sigma^l_{2,2} =  \left(8 p^2-6 p-1\right) \alpha^3+\left(48 p^2-46 p+1\right) \alpha^2 \\
& \hspace{20mm} +\left(112 p^2-158 p+49\right) \alpha+88 p^2-158 p+71 \\
\sigma^l_{1,2} &= \sigma^l_{2,1} = \left(1 + 6p -8 p^2\right) \alpha^3-\left(48 p^2-50p+7\right) \alpha^2\\
& \hspace{20mm}-\left(96 p^2-126 p+37\right) \alpha-72 p^2+122p-53
\end{align*}
when $p \neq 1/2$, and 
\[ \Sigma_I^l = \frac{\alpha+1}{4(2+\alpha)^2(3+\alpha)}
\left(\begin{array}{cc}
7 + 6\alpha + \alpha^2 & -5-4\alpha-\alpha^2 \\
-5-4\alpha-\alpha^2 & 7 + 6\alpha + \alpha^2
\end{array}\right)\]
when $p = 1/2$. 
\item[(iii)]  If $p = (3-\alpha)/4$ and $p \neq \frac12$ (i.e., $\alpha \neq 1$), then the following multivariate normal limit law holds
\begin{equation*}\label{eq:leaves=}
 \frac{(R_n^l, B_n^l) - n\left(\frac{1+\alpha}{4+2\alpha}, \frac{1+\alpha}{4+2\alpha}\right)}{\sqrt{n\ln{n}} }\xrightarrow{d} \mathcal{N}(\bm{0} , \Sigma^l_{II}), 
 \end{equation*}
where 
\[ \Sigma^l_{II} = \frac{(\alpha-1)^2(\alpha+1)}{4(3+\alpha)^2}
\left( \begin{array}{cc}
1 & -1 \\
-1 & 1 
\end{array}\right) .\]

\item[(iv)] If $p > (3-\alpha)/4$ and $p \neq \frac12$, then the following convergence in distribution holds
\begin{equation*}\label{eq:leaves>}
\frac{(R_n^l, B_n^l)- n\left(\frac{1+\alpha}{4+2\alpha}, \frac{1+\alpha}{4+2\alpha}\right)}{n^{(2p+\alpha - 1)/(1 + \alpha)}} \xrightarrow{d} \frac{B Z}{2\alpha + 4p}(2p-1,1-2p),
\end{equation*}
where $Z$ is the same random variable as in \eqref{eq:vts>}.
\end{enumerate}
\end{theorem}

\begin{remark}\label{rem:lvs,p=1/2}
Once more, we witness a special case $p=1/2$. Though the explanation is not as simple as in Remark \ref{rem:vts,p=1/2}, the colour of a newly added leaf is independent of the colour of its parent. We see again that the random variables on the right hand side of the convergences in (iii) and (iv) degenerate to (0,0). When $p=1/2 \neq (3-\alpha)/4$, both matrices in $(ii)$ are equal. 
\end{remark}

Limiting joint distributions for the number of fringe subtrees (without colours) have already been studied \cite{HOJS:17}. Let $T_1, \ldots, T_m$ be a sequence of finite trees of sizes $k_1, \ldots, k_m$ with colourings $\varsigma_1, \ldots, \varsigma_m$. Let $\left.\sigma_n\right|_{T}$ denote the colouring $\sigma_n$ restricted to the subtree $T$. We say that two coloured rooted trees are isomorphic if there is an isomorphism between them that preserves roots and colours.

\begin{theorem}\label{thm:fringe}
Let $X_n^{i}$ be the number of fringe subtrees $T$ in $\mathcal{T}_n$ isomorphic to $T_i$ with colouring $\varsigma_i$, and let
\begin{equation}\label{eq:mu}
 \bm{\mu} = \left( \frac{\mathbb{P}((\mathcal{T}_{k_1}, \sigma_{k_1}) \simeq (T_1, \varsigma_1))\frac{1}{\alpha + 1}}{(k_1 + \frac{1}{\alpha + 1}- 1)(k_1 + \frac{1}{\alpha + 1})}, \ldots, \frac{\mathbb{P}((\mathcal{T}_{k_m}, \sigma_{k_m}) \simeq (T_m, \varsigma_m))\frac{1}{\alpha + 1}}{(k_m + \frac{1}{\alpha + 1} - 1)(k_m + \frac{1}{\alpha + 1})}\right).
\end{equation}
\begin{enumerate}
\item[(i)] The following strong law of large numbers holds
\[ \frac{1}{n}\left(X^{1}_n, \ldots, X^{m}_n\right) \xrightarrow{a.s} \bm{\mu}.\]
\item[(ii)] If $p < (3-\alpha)/4$ or if $p=1/2$, then the following multivariate normal limit law holds
\[ \frac{\left(X^{1}_n, \ldots, X^{m}_n\right) - n\bm{\mu}}{\sqrt{n}} \xrightarrow{d} \mathcal{N}(\bm{0}, \Sigma_I^f)\]
for some covariance matrix $\Sigma_I^f$.
\item[(iii)] If $p = (3-\alpha)/4$ and $p \neq \frac{1}{2}$ (i.e., $\alpha \neq 1$), then the following multivariate normal limit law holds
\[ \frac{\left(X^{1}_n, \ldots, X^{m}_n\right) - n\bm{\mu}}{\sqrt{n\ln n}} \xrightarrow{d} \mathcal{N}(\bm{0}, \Sigma_{II}^f)\]
for some covariance matrix $\Sigma_{II}^f.$
\item[(iv)] If $p > (3-\alpha)/4$ and $p \neq \frac{1}{2}$, then the following convergence in distribution holds
\[\frac{\left(X^{1}_n, \ldots, X^{m}_n\right) - n\bm{\mu}}{n^{(2p+\alpha-1)/(1+\alpha)}} \xrightarrow{d} B Z\bm{v}\]
for some vector $\bm{v}$, where $Z$ is the same random variable as in \eqref{eq:vts>}.
\end{enumerate}
\end{theorem}

\begin{remark}
The matrices $\Sigma^f_I$, $\Sigma^f_{II}$, as well as the vector $\bm{v}$ can be calculated explicitly from the sequence $T_1, \ldots, T_m$ of trees and the colourings $\varsigma_1, \ldots, \varsigma_m$ (see Theorem \ref{thm:urns} below).
\end{remark}

\begin{remark}
The statements of Theorem \ref{thm:vts} (iv), Theorem \ref{thm:clusters} (iv), Theorem \ref{thm:leaves} (iv), and Theorem \ref{thm:fringe} (iv), all contain the same random variable $BZ$ in the limit. More precisely, as is made evident by the proofs in Section \ref{sec:global}, the sequences of random vectors in all of these statements converge jointly. This is proved by studying a P\'olya urn process (see Proposition \ref{prop:weights}) which is a linear transformation of each of the P\'olya urn processes used in the proofs of Theorems \ref{thm:vts}, \ref{thm:clusters}, \ref{thm:leaves}, and \ref{thm:fringe}, and applying the Cram\'er-Wold Theorem \cite[ch. 5, Theorem 10.5]{GUT:13}.
\end{remark}

\subsection{Root cluster}\label{sec:resultsrootcluster}

We let $\mathcal{C}_n$ denote the {\em root cluster} in the random preferential attachment tree $\mathcal{T}_{n}$ with broadcasting induced colouring $\sigma_n$, that is, the maximal monochromatic subtree containing the root. As noted in the introduction, $\mathcal {C}_n$ is identically distributed as the root cluster in $\mathcal{T}_{n}$ after applying Bernoulli bond percolation with parameter $p$. 

The first theorem was previously proved by Möhle \cite{MOHL:15}, Baur and Bertoin \cite{BABE:15}, and Businger \cite{BUSI:18}.
\begin{theorem}\label{thm:rootclustera=0}
Let $\alpha = 0$. Then 
\[ \frac{|\mathcal{C}_n|}{n^p} \xrightarrow{d} \mathcal{C},\]
where $\mathcal{C}$ is a Mittag-Leffler distribution with parameter $p$; that is, $\mathcal{C}$ is characterized by its integer moments 
\[ \mathbb{E}[\mathcal{C}^k] = \frac{k!}{\Gamma(pk+1)}.\]
\end{theorem}

Baur \cite{BAUR:20} proved convergence in distribution for $|\mathcal{C}_n|/n^{(p+\alpha)/(1+\alpha)}$ when $\alpha >0$ (though we believe the method may also apply for applicable $\alpha > -1/p$). We extend the results by finding a recursion for the integer moments of the limiting distribution. This recursion uses the partial Bell polynomials (see~\cite[Chapter 3.3]{COMT:74})
\[ B_{k,j}(x_1, \ldots, x_{k-j+1}) = \sum_{\substack{m_1 + \cdots + (k-j+1)m_{k-j+1} = k\\ m_1 + \cdots + m_{k-j+1} = j}}k!\prod_{i=1}^{k-j+1} \frac{x_i^{m_i}}{m_i!i!^{m_i}}.\]

\begin{theorem}\label{thm:rootclustera>0}
Let $\alpha > 0$. Then 
\[ \frac{|\mathcal{C}_n|}{n^{(p+\alpha)/(1 + \alpha)}} \xrightarrow{d} \mathcal{C},\]
where $\mathcal{C}$ has integer moments 
\[ \mathbb{E}[\mathcal{C}^k] = \frac{C_k(1 + \alpha)\Gamma(1/(1 + \alpha))}{\alpha\Gamma((kp + \alpha(k-1))/(\alpha + 1))},\]
where $C_k$ satisfies the recursion $C_1 = \alpha/(p + \alpha)$ and 
\[ (k-1)(p/\alpha +1)C_k = \sum_{j=2}^k \frac{p^j \Gamma(1/\alpha + j)}{\Gamma(1/\alpha)}B_{k,j}(C_1, \ldots, C_{k-j+1}).\]
\end{theorem}

By using the recursion in Theorem \ref{thm:rootclustera>0}, we calculate the first two moments of $\mathcal{C}$ to be 
\begin{align*}
\mathbb{E}[\mathcal{C}] &= \frac{(1+\alpha)\Gamma\left(\frac{1}{1+\alpha}\right)}{(p + \alpha)\Gamma\left(\frac{p}{1 + \alpha}\right)},\\
\mathbb{E}[\mathcal{C}^2] &= \frac{p^2(1 + \alpha)^2\Gamma\left(\frac{1}{1+\alpha}\right)}{(p + \alpha)^3 \Gamma\left(\frac{2p + \alpha}{1 + \alpha}\right)}, 
\end{align*}
which agrees with the calculations given in \cite{BAUR:20}.

Baur provides a description of the sizes of the remaining clusters \cite[Corollary 4.3]{BAUR:20}. If the $i$'th added vertex is the root of a cluster, the size of this cluster, scaled by $n^{(p + \alpha)/(1 + \alpha)}$, converges in distribution to 
\[ \beta_i^{(p + \alpha)/(1 + \alpha)}\mathcal{C},\]
where $\mathcal{C}$ is the random variable given in Theorem \ref{thm:rootclustera>0} and $\beta_i$ is a Beta$(1/(1 + \alpha), i)$ distributed random variable.

In the special case $\alpha = 1$, we are able to find a closed form for the recursion given in Theorem \ref{thm:rootclustera>0}, and with it, a more precise description of the limiting distribution of $|\mathcal{C}_n|$ after proper rescaling. 

\begin{theorem}\label{thm:rootclustera=1}
Let the underlying tree be a random plane-oriented recursive tree, so $\alpha = 1$. Then the integer moments of the limiting distribution $\mathcal{C}$ in Theorem \ref{thm:rootclustera>0} can be written as 
\[ \mathbb{E}[\mathcal{C}^k] = \frac{2p^{k-1}\Gamma(kp + k -1)\sqrt{\pi}}{(p+1)^{2k-1}\Gamma(kp)\Gamma((kp + (k-1))/2)}.\]
\end{theorem}

We now turn to the case when $\alpha = -1/d$ for an integer $d \geq 2$, that is, when the underlying tree $\mathcal{T}_n$ is a random increasing $d$-ary tree.  
If we consider $\mathcal{T}_n$ as a subtree of an infinite $d$-ary tree $T_d$, then the colouring $\sigma_n$ can be recovered from bond percolation on $T_d$: start by assigning the root either red or blue, and assign to a vertex $v$ the colour of its parent $u$ if the edge joining $u$ and $v$ is still present after performing bond percolation, and the other colour otherwise. In this way, the root cluster $\mathcal{C}_n$ of $\mathcal{T}_n$ is a subtree of the cluster $\mathcal{K}_d$ of $T_d$ containing the root after performing bond percolation. Using this fact, we can prove the following result on the sizes of $|\mathcal{C}_n|$ and $|\mathcal{K}_d|$ (which may be infinite). 

\begin{theorem}\label{thm:rootclusterba.s.}
Let $|\mathcal{C}_n|$ be the size of the root cluster of $\mathcal{T}_n$ with broadcasting induced colouring $\sigma_n$. Then $|\mathcal{C}_n| \xrightarrow{a.s.} |\mathcal{K}_d|.$
\end{theorem}

 Using well known results on the size of $|\mathcal{K}_d|$, the following corollary is immediate:

\begin{corollary}\label{thm:rootclusterbsub}
Let $\alpha = -1/d$, where $d\geq 2$ is a positive integer. Then for every positive integer $k$, 
\begin{equation}\label{eq:OttDwa}
\mathbb{P}\left(\lim_{n\to \infty} |\mathcal{C}_n| = k\right) = \frac{1}{k}\binom{kd}{k-1}p^{k-1}(1-p)^{kd-k+1}.
\end{equation}
\end{corollary}

\begin{remark}
When $p \leq 1/d$, the probabilities in \eqref{eq:OttDwa} sum to 1, and so the root cluster is almost surely finite. 
\end{remark}

When $\alpha = -p$, the root cluster is almost surely finite, though its expected size grows to infinity. In fact, we can describe the asymptotic behaviour of all the moments of $|\mathcal{C}_n|$.

\begin{theorem}\label{thm:rootclusterbcrit}
Let $\alpha = -1/d$, where $d\geq 2$ is a positive integer, and let $p = -\alpha = 1/d$. Then 
\[ \mathbb{E}[|\mathcal{C}_n|^k] \sim E_k\ln^{2k-1}n,\]
where $E_k$ satisfies the recursion $E_1 = 1/(d-1)$ and 
\begin{equation*}
(2k-1)E_k = \frac{1}{2d}\sum_{j=1}^{k-1}\binom{k}{j}E_jE_{k-j}.
\end{equation*}
\end{theorem}

In the case $\alpha > -p$, a similar limiting distribution $\mathcal{C}$ to that found in Theorem \ref{thm:rootclustera>0} exists.

\begin{theorem}\label{thm:rootclusterbsup}
Let $\alpha = -1/d$, where $d\geq 2$ is a positive integer, and let $p > -\alpha = 1/d$. Then 
\[ \frac{|\mathcal{C}_n|}{n^{(pd-1)/(d-1)}} \xrightarrow{d} \mathcal{C},\]
where $\mathcal{C}$ has integer moments 
\[ \mathbb{E}[\mathcal{C}^{k}] =  \frac{D_k\Gamma(1/(d-1))}{\Gamma((kpd - k + 1)/(d-1))} ,\]
where $D_k$ satisfies the recursion $D_1 = 1/(pd-1)$ and 
\[ (k-1)(pd-1)D_k = \sum_{j=2}^{\min\{k,d\}} \frac{p^j d!}{(d-j)!} B_{k,j}(D_1, \ldots, D_{k-j+1}).\]
\end{theorem}

By using the recursion in Theorem \ref{thm:rootclusterbsup}, we calculate the first two moments of $\mathcal{C}$ to be
\begin{align*}
\mathbb{E}[\mathcal{C}] &= \frac{\Gamma\left(\frac{1}{d-1}\right)}{(pd-1)\Gamma\left(\frac{pd}{d-1}\right)}, \\
\mathbb{E}[\mathcal{C}^2] &= \frac{p^2d(d-1)\Gamma\left(\frac{1}{d-1}\right)}{(pd-1)^3\Gamma\left(\frac{2pd-1}{d-1}\right)}.
\end{align*}

When the underlying tree is a binary search tree (when $d=2$), we can once again find a complete description of the limiting distribution. 

\begin{theorem}\label{thm:rootclusterb=2}
Let the underlying tree be a random binary search tree, so $d=2$, and let $p > -\alpha = 1/2$. Then the integer moments of the limiting distribution $\mathcal{C}$ in Theorem \ref{thm:rootclusterbsup} can be written as 
\[  \mathbb{E}[\mathcal{C}^k] = \frac{k!p^{2(k-1)}}{(2p-1)^{2k-1}\Gamma(k(2p-1) + 1)}.\]
\end{theorem}

\begin{remark}
From Theorem 2.12, there is a positive probability that the size of the root cluster is finite, even in the case $p > 1/d$. For any $p$ and $d$, let $p_\infty$ be the smallest positive solution to 
\[ 1-x = (1-px)^d.\]
It is known that $p_\infty$ is the probability that the cluster $\mathcal{K}_d$ containing the root after performing Bernoulli bond percolation on an infinite $d$-ary tree is infinite (see for example \cite[Exercise 5.41]{LYPE:16}). From Theorem \ref{thm:rootclusterba.s.}, $|\mathcal{C}_n|$ converges almost surely to $|\mathcal{K}_d|$. The limiting random variable $\mathcal{C}$ in Theorem \ref{thm:rootclusterbsup} can be broken down in the following way:
\[ \mathcal{C} = \begin{cases}
0 & \text{with probability } 1- p_\infty, \\
\mathcal{C}_\infty & \text{with probability } p_\infty,
\end{cases}\]
where $C_{\infty}$ has moments 
\[ \mathbb{E}[\mathcal{C}_\infty^k] = p_\infty^{-1}\mathbb{E}[\mathcal{C}^k].\]
When $d=2$, then $p_\infty = (2p-1)/p^2$, and $\mathcal{C}_\infty$ has moments 
\[ \mathbb{E}[\mathcal{C}_\infty^k]= \frac{k!p^{2k}}{(2p-1)^{2k}\Gamma(k(2p-1)+1)},\]
which are the moments of a generalized Mittag-Leffler distribution.
\end{remark}

\section{Proofs of global properties}\label{sec:global}


We start by summarizing some results on generalized P\'olya urns that will be used throughout this section. 
A generalized P\'olya urn process $(\bm{X}_n)_{n=0}^{\infty}$ is defined as follows. There are $q$ types (or colours) $1,2, \ldots, q$ of balls, and for each vector $\bm{X}_n = (X_{n,1}, X_{n,2}, \ldots, X_{n,q})$, the entry $X_{n,i} \geq 0$ is the number of balls of type $i$ in the urn at time $n$. For each type $i$, an activity $a_i \geq 0$ is assigned, as well as a random vector $\bm{\xi}_i = (\xi_{i,1}, \xi_{i,2}, \ldots, \xi_{i,q})$ such that $\xi_{i,j} \geq 0$ for $i \neq j$ and $\xi_{i,i} \geq -1$. The urn process begins with a given vector $\bm{X}_0$. At time $n \geq 1$, a ball is drawn uniformly at random from the urn, so that the probability that a ball of colour $i$ is chosen is
\[ \frac{a_iX_{n-1,i}}{\sum_{j=1}^q a_jX_{n-1,j}}.\]
If the drawn ball is of type $i$, then we set $\bm{X}_n = \bm{X}_{n-1} + \Delta \bm{X}_n$, where $\Delta \bm{X}_n \sim \bm{\xi}_i$ and is independent of everything that has happened so far. 
The {\em intensity matrix} of the P\'olya urn is the $q \times q$ matrix 
\[ A := \left(a_j\mathbb{E}[\xi_{j,i}]\right)_{i,j =1}^q.\]
Note that while several authors place $\mathbb{E}[\bm{\xi_i}]$ for row $i$ of $A$, we follow the notation of \cite{JANS:04} by placing $\mathbb{E}[\bm{\xi_j}]$ for column $j$ of $A$. As noted in \cite{JANS:04}, since all off diagonal entries of $A$ are non-negative, $A$ has a largest real eigenvalue $\lambda_1$ such that $\lambda_1 > \text{Re}\lambda$ for all other eigenvalues $\lambda$ of $A$. A type $i$ is called {\em dominating} if for all other types $j$, it is possible to find a ball of type $j$ in an urn beginning with a single ball of type $i$. By ordering the types such that every dominating $i$ is smaller than every nondominating type $j$, the matrix $A$ will be a block diagonal matrix. We say that an eigenvalue $\lambda$ of $A$ belongs to the dominating class if it is also an eigenvalue of the submatrix of $A$ restricted to the dominating types. 

The following six assumptions appear in \cite{JANS:04} (the assumption (A1) is a generalization from \cite[Remark 4.2]{JANS:04}, note the indices of the variables in (A1)):
\begin{enumerate}
\item[(A1)] For each $i=1,\ldots,q$, either 
\begin{enumerate}
\item there is a real number $d_i > 0$ such that $X_{0,i}$ and $\xi_{1,i}, \xi_{2,i}, \ldots, \xi_{q,i}$ are multiples of $d_i$ and $\xi_{i,i} \geq -d_i$, or 
\item $\xi_{i,i} \geq 0$.
\end{enumerate}
\item[(A2)] $\mathbb{E}[\xi_{i,j}^2] < \infty$ for all $i,j = 1,\ldots,q$.
\item[(A3)] The largest real eigenvalue $\lambda_1$ of $A$ is positive. 
\item[(A4)] The largest real eigenvalue $\lambda_1$ is simple.
\item[(A5)] There exists a dominating type $i$ with $X_{0,i} > 0$.
\item[(A6)] $\lambda_1$ belongs to the dominating class.
\end{enumerate}
We add the following simplifying assumption
\begin{enumerate}
\item[(A7)] For each $n\geq 1$ there exists a ball of dominating type in the urn.
\end{enumerate}
We further add the following assumption which will make the covariance matrix calculations simpler 
\begin{enumerate}
\item[(A8)] There exists $c>0$ such that $\sum_{i=1}^q a_i \mathbb{E}[\xi_{j,i}] = c$ for every $j=1, \ldots, q$ where $a_j > 0$. 
\end{enumerate}

All vectors $\bm{v}$ for the remainder of this discussion are assumed to be column vectors. Let $\bm{a} = (a_1, \ldots, a_q)^T$ be the vector of activities. Let $\bm{v}_1$ and $\bm{u}_1$ be the right and left eigenvectors associated with $\lambda_1$ normalized such that $\bm{a}^T \bm{v}_1 = 1$ and $\bm{u}_1^T \bm{v}_1 = 1$. 
Order the eigenvalues $\lambda_1, \lambda_2, \ldots, \lambda_q$ such that $\lambda_1 \geq \text{Re}\lambda_2 \geq \text{Re}\lambda_3 \geq \cdots \geq \text{Re}\lambda_q$. If $A$ is diagonalizable, then there are $q$ linearly independent right eigenvectors of $A$ and $q$ linearly independent left eigenvectors of $A$. Let $\bm{v}_i$ and $\bm{u}_i^T$ be dual bases for the eigenspaces of $A$, that is, right and left eigenvectors of $A$ associated with $\lambda_i$ such that $\bm{u}_i^T\bm{v}_j = \delta_{i,j}$ for all $i,j=1,\ldots,q$, where
\[ \delta_{i,j} = 
\begin{cases}
1 & i=j, \\
0 & i\neq j.
\end{cases}\]
Denote $\bm{v}_1 =: (v_{1,1}, v_{1,2}, \ldots, v_{1,q})^T$ 
and define the matrices
\[ B:= \sum_{i=1}^q a_iv_{1,i}\mathbb{E}\left[\bm{\xi}_i\bm{\xi_i}^T\right]\]
and 
\begin{equation*}
 \Sigma_I =  \sum_{j,k=2}^q \frac{\bm{u}_j^T B \bm{u}_k}{\lambda_1 - \lambda_j - \lambda_k}\bm{v}_j \bm{v}_k^T
 \end{equation*}
whenever none of the denominators is equal to zero (which holds in the cases relevant to us). Let $P = I - \bm{v}_1\bm{u}_1$ and 
\begin{equation}\label{eq:Sigma1special}
\Sigma_I^\dagger = \int_0^\infty Pe^{sA}Be^{sA^T}P^Te^{-\lambda_1 s}ds.
\end{equation}
If $\lambda_2$ is real and $\lambda_2 > \text{Re}\lambda_3$, then define the matrix
\begin{equation}\label{eq:Sigma2}
 \Sigma_{II} := (\bm{u}_2^T B \bm{u}_2)\bm{v}_2\bm{v}_2^T.
 \end{equation}

We are now ready to gather results from \cite{JANS:04}. 
\begin{theorem}[Janson 2004, \cite{JANS:04}]\label{thm:urns}
Suppose an urn process $(\bm{X}_n)_{n=0}^{\infty}$ satisfies (A1)-- (A7). The following hold:
\begin{enumerate}
\item[(i)] a strong law of large numbers, \[ \frac{\bm{X}_n}{n} \xrightarrow{a.s.} \lambda_1 \bm{v}_1,\]
\item[(ii)] if (A8) is satisfied and $\lambda_1 > 2\text{Re}\lambda_2$, then
\[ \frac{\bm{X}_n - n\lambda_1 \bm{v}_1}{\sqrt{n}} \xrightarrow{d} \mathcal{N}(\bm{0}, \lambda_1\Sigma_I^\dagger),\]
where $\mathcal{N}$ denotes a multivariate normal distribution, and $\Sigma_I^\dagger$ is defined as in \eqref{eq:Sigma1special}. If $A$ is diagonalizable, then $\Sigma_I^\dagger$ can be replaced with $\Sigma_I$,
\item[(iii)] if (A8) is satisfed, $\lambda_1 = 2 \lambda_2$, $\lambda_2 > \text{Re}\lambda_3$, and $A$ is diagonalizable, then
\[  \frac{\bm{X}_n - n\lambda_1 \bm{v}_1}{\sqrt{n\ln{n}}} \xrightarrow{d} \mathcal{N}(\bm{0},\Sigma_{II}),\]
where $\mathcal{N}$ denotes a multivariate normal distribution, and $\Sigma_{II}$ is defined as in \eqref{eq:Sigma2},
\item[(iv)] if $\lambda_2$ is real, $\lambda_1 < 2\lambda_2$, and $\lambda_1 > 2\text{Re}\lambda_i$ for all $i=3, \ldots, q$, then 
\[ \frac{\bm{X}_n - n\lambda_1 \bm{v}_1}{n^{\lambda_2 / \lambda_1}} \xrightarrow{a.s} \widehat{Z}\bm{v}_2.\]
where $\widehat{Z}$ is a real random variable.
\end{enumerate}
\end{theorem}

\begin{proof}
The convergence in (i) follows from \cite[Theorem 3.21]{JANS:04} (essential non-existence is always guaranteed if (A7) holds), the convergence in (ii) follows from \cite[Theorem 3.22]{JANS:04}, and the convergence in (iii) follows from \cite[Theorem 3.23]{JANS:04}, while the covariance matrix calculations in (ii) and (iii) follow from \cite[Lemma 5.3(i), Lemma 5.4]{JANS:04}, where we note that the proof of \cite[Lemma 5.4]{JANS:04} follows exactly the same with the slightly more general assumption (A8). The convergence in (iv) follows from \cite[Theorem 3.24]{JANS:04} by letting $\widehat{Z} = \bm{u}_2^T \widehat{W}$ (the random vector $\widehat{W}$ is an element of the eigenspace of $\lambda_2$).
\end{proof}


We are now ready to prove our results of global properties for $\mathcal{T}_n := \mathcal{T}_{\alpha,n}$ with broadcasting induced colouring $\sigma_n := \sigma_{\mathcal{T}_n,p}$. Let $\alpha\deg^+(v) + 1$ be the {\em weight} of the vertex $v$ in $\mathcal{T}_n$. We consider an urn with two colours of balls: red $r$ and blue $b$, both with activity 1. In this urn, the total activity of red and blue balls at time $n$ will correspond to the sum of the total weights of red and blue vertices in the tree $\mathcal{T}_{n}$ with colouring $\sigma_{n}$, respectively. When a ball is picked, with probability $p$ it is replaced with an additional $1 + \alpha$ balls of the same colour; $1$ corresponding to the addition of a new vertex, while the extra $\alpha$ corresponds to the increase in weight of the selected vertex. With probability $1-p$, the chosen ball is replaced along with $\alpha$ balls of the same colour (corresponding to the increase in weight), while an additional $1$ ball of the other colour is added (corresponding to the new vertex added). Let $R^w_n$ and $B^w_n$ be the total activity of red and blue balls respectively at time $n$, which is also the total weight of the vertices of each colour in $\mathcal{T}_n$. We therefore have the following activity matrix for our urn: 
\begin{center}
\begin{tabular}{ccc}
& $\begin{array}{cc} r \,\,\, &\,\,\, b\end{array}$ & \\
 $A =$ & $\left(\begin{array}{cc}
\alpha + p & 1-p \\
1-p & \alpha + p
\end{array} \right) $ & 
$\begin{array}{c} r \\ b \end{array}$
\end{tabular}
\end{center}

This particular P\'olya urn process was previously studied in the context of preferential attachment trees by Baur and Bertoin to study elephant random walks \cite{BABE:16}. The eigenvalues of $A$ are $\lambda_1 = 1+\alpha$ and $\lambda_2 = 2p  + \alpha-1$, while $A$ satisfies (A1)--(A8). Therefore, Theorem \ref{thm:urns} applies with $\bm{v}_1 = (1/2, 1/2), \bm{u}_1 = (1,1), \bm{v}_2 = (1/2, -1/2)$, and $\bm{u}_2 = (1,-1)$. 

We can say something more about the limiting distribution in this case when $2\lambda_2 > \lambda_1$ (so when $p>(3-\alpha)/4$). Recall the random variable $B$ defined in \eqref{eq:rootcolour} ($B = 1$ if the root is red and $B=-1$ if the root is blue, so $B \sim (-1)^{\text{Be}(1/2)}$ where $\text{Be}$ denotes a Bernoulli random variable). 
\begin{proposition}\label{prop:weights}
Let $R_n^w$ and $B_n^w$ be the total weight of red and blue balls respectively, and suppose that $p > (3-\alpha)/4$. Then 
\begin{equation}\label{eq:weights}
 \frac{(R_n^w, B_n^w) - \left(\frac{1}{2}, \frac{1}{2}\right)}{n^{(2p+\alpha -1)/(1+\alpha)}} \xrightarrow{d} BZ\left(\frac{1}{2}, -\frac{1}{2}\right),
\end{equation}
where $Z$ is a real random variable with 
\[ \mathbb{E}[Z] = \frac{\Gamma(1/(1 + \alpha))}{\Gamma((2p + \alpha)/(1 + \alpha))},\]
and 
\[\mathbb{E}[Z^2]  = \frac{\Gamma(1/(1 + \alpha))(1 + \alpha)(4p + \alpha-2)}{\Gamma((4p + 2\alpha -  1)/(1 + \alpha))(4p + \alpha - 3)}.\]
\end{proposition}
\begin{proof}
First, suppose we always start with a red root (so start the urn with a red ball). Then the convergence in \eqref{eq:weights} with $B=1$ follows from \ref{thm:urns} (iv). For the calculation of the expected value and the variance (again assuming we start with a red ball), we appeal to \cite[Theorem 3.10, Theorem 3.26]{JANS:04}.  The random variable $Z_1 =  \bm{u}_2^TW_{\lambda_2,1}$ corresponds in this case to starting with a single ball of colour $r$, and $Z_2 = \bm{u}_2^TW_{\lambda_2, 2}$ corresponds to starting with a single ball of colour $b$. The expected value of $Z_1$ is the first component of $u_2$. By symmetry, $\sigma_1^2 = \textrm{Var}[Z_1] = \textrm{Var}[Z_2] = \sigma_2^2$, and so 
\begin{align*}
(2\lambda_2 - (\alpha+p) - (1-p))\sigma_1^2 &= \lambda_2^2 + \mathbb{E}[(\bm{u}_2^T\xi_1)^2] - (\bm{u}_2^T\mathbb{E}[\xi_1])^2 \\
&= p(1+\alpha)^2 + (1-p)(\alpha - 1)^2.
\end{align*}
Rearranging for $\sigma_1^2$ and adding $(\mathbb{E}[Z_1])^2 = 1$ gives 
\[ \mathbb{E}[Z_1^2] = \frac{(1+\alpha)(4p+\alpha - 2)}{4p + \alpha -3}.\]
Then by applying \cite[Eq. 3.21]{JANS:04}, we get 
\[ \mathbb{E}[Z] = \frac{\Gamma(1/(1+\alpha))}{\Gamma((2p+\alpha)/(1+\alpha))}\mathbb{E}[Z_1] =  \frac{\Gamma(1/(1+\alpha))}{\Gamma((2p+\alpha)/(1+\alpha))},\]
and 
\[ \mathbb{E}[Z^2] = \frac{\Gamma(1/\lambda_1)}{\Gamma((1 + 2\lambda_2)/\lambda_1)}\mathbb{E}[Z_1^2] = \frac{\Gamma(1/(1 + \alpha))(1 + \alpha)(4p + \alpha-2)}{\Gamma((4p + 2\alpha -  1)/(1 + \alpha))(4p + \alpha - 3)}.\]
Next we multiply by $B$ since the urn starts with a single red ball with probability $1/2$, and a single blue ball with probability $1/2$. 
\end{proof}

While a P\'olya urn can be used to study the number of vertices of each colour, a simpler proof follows from limit laws for the number of clusters (maximal monochromatic subtrees) of each colour. We therefore start with studying the clusters in $\mathcal{T}_n$. 


\begin{proof}[Proof of Theorem \ref{thm:clusters}]
Consider an urn with four colours of balls: $r, b$ with activity 1, and $r^c, b^c,$ with activity 0. Let $R^w_n, B^w_n$ be the total number of balls (and so the total activity of the balls) of colour $r, b$, respectively, and let $R^c_n, B^c_n$ be the number of balls of colour $r^c, r^b$ respectively. As in the urn above, the balls $r$ and $b$ represent the weights of the red and blue vertices in $\mathcal{T}_{n}$ with colouring $\sigma_{\mathcal{T}_{n}}$. The balls of colours $r^c$ and $b^c$ represent clusters of colour red and blue respectively. 
We start the urn with a ball of colour $r$ and a ball of colour $r^c$ if the root is red, and a ball of colour $b$ and a ball of colour $b^c$ if the root is blue.
Therefore the number of red and blue clusters at time $n$ is exactly $R_n^c$ and $B_n^c$ respectively. 

For example in Figure \ref{fig:clusters}, there are 7 red clusters and 5 blue clusters, so $R_n^c =7$ and $B_n^c = 5$. Each vertex $v$ contributes $\alpha \deg^+(v) + 1$ to the total weight of its colour. Summing over all red vertices yields $R_n^w = 13 + 11\alpha$ and summing over all blue vertices yields $B_n^w = 10 + 11\alpha$. 

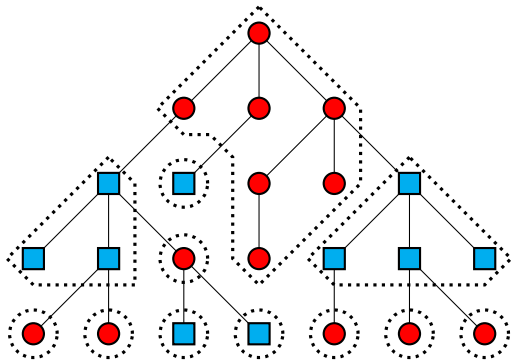
\begin{figure}[h!]
\centering
\begin{tikzpicture}[
redv/.style={circle, inner sep=0pt, draw=black, fill=red, thick, minimum size=8pt},
bluv/.style={rectangle, inner sep=0pt, draw=black, fill=cyan, thick, minimum size=8pt}
]

\draw (5,4) -- (4,3) -- (3,2) -- (2,1) ;
\draw (5,4) -- (6,3) -- (7,2) -- (8,1);
\draw (5,4) -- (5,3) -- (4,2);
\draw (6,3) -- (5,2) -- (5,1);
\draw (6,3) -- (6,2);
\draw (3,2) -- (3,1) -- (2,0);
\draw (3,2) -- (4,1) -- (4,0);
\draw (7,2) -- (6,1) -- (6,0);
\draw (7,2) -- (7,1) -- (7,0);
\draw (3,1) -- (3,0);
\draw (4,1) -- (5,0);
\draw (7,1) -- (8,0);

\node[redv] at (5,4){};

\node[redv] at (4,3){};
\node[redv] at (5,3){};
\node[redv] at (6,3){};

\node[bluv] at (3,2){};
\node[bluv] at (4,2){};
\node[redv] at (5,2){};
\node[redv] at (6,2){};
\node[bluv] at (7,2){};

\node[bluv] at (2,1){};
\node[bluv] at (3,1){};
\node[redv] at (4,1){};
\node[redv] at (5,1){};
\node[bluv] at (6,1){};
\node[bluv] at (7,1){};
\node[bluv] at (8,1){};

\node[redv] at (2,0){};
\node[redv] at (3,0){};
\node[bluv] at (4,0){};
\node[bluv] at (5,0){};
\node[redv] at (6,0){};
\node[redv] at (7,0){};
\node[redv] at (8,0){};

\draw[dotted, very thick] (5, 4.35) -- (3.65,3) -- (4,2.65) -- (4.35, 2.65) -- (4.65, 2.35) -- (4.65, 1) -- (5, 0.65) -- (6.35, 2) -- (6.35, 3) -- (5, 4.35);
\draw[dotted, very thick] (7,2.35) -- (5.65, 1) -- (6, 0.65) -- (8, 0.65) -- (8.35, 1) -- (7, 2.35);
\draw[dotted,  very thick] (3, 2.35) -- (1.65, 1) -- (2, 0.65) -- (3.35, 0.65) -- (3.35, 2) -- (3, 2.35);

\node[circle,draw, dotted, very thick,minimum size = 18pt]  at (4,2){};
\node[circle,draw, dotted, very thick,minimum size = 18pt]  at (4,1){};
\node[circle,draw, dotted, very thick,minimum size = 18pt]  at (2,0){};
\node[circle,draw, dotted, very thick,minimum size = 18pt]  at (3,0){};
\node[circle,draw, dotted, very thick,minimum size = 18pt]  at (4,0){};
\node[circle,draw, dotted, very thick,minimum size = 18pt]  at (5,0){};
\node[circle,draw, dotted, very thick,minimum size = 18pt]  at (6,0){};
\node[circle,draw, dotted, very thick,minimum size = 18pt]  at (7,0){};
\node[circle,draw, dotted, very thick,minimum size = 18pt]  at (8,0){};

\end{tikzpicture}
\caption{A tree $\mathcal{T}_{23}$ with broadcasting induced colouring $\sigma_{23}$ with the clusters identified.}
\label{fig:clusters}
\end{figure}

If a red vertex is chosen at step $n$, this corresponds to choosing a ball of colour $r$. Then with probability $p$ it is replaced with an additional $1+\alpha$ balls of colour $r$, just as above. With probability $1-p$ however, the ball is replaced along with $\alpha$ balls of colour $r$ along with 1 ball of colour $b$ (just as above), with an addition ball of colour $b^c$ added representing the new blue cluster that is formed. Therefore, we have 
\[ \mathbb{E}[\xi_r] = (\alpha + p, 1-p, 0, 1-p)^T,\]
the first column in the intensity matrix $A^c$ for this urn. The symmetric argument for balls of colour $b$ holds, contributing to the second column of $A^c$. Balls of colour $r^c$ and $b^c$ have activity 0, and so the intensity matrix for this urn is 
\begin{center}
\begin{tabular}{ccc}
&$\begin{array}{cccc}\,\,\,\,\,\,\,\,\, r\,\,\,\, & \,\,\,\,\,\,b \,\,& \,\,r^c &b^c \end{array}$ & \\
$A^c =$ & $ \left(\begin{array}{cccc}
\alpha + p & 1-p & 0 & 0 \\
1-p & \alpha + p & 0 & 0 \\
0 & 1-p & 0 & 0 \\
1-p & 0 & 0 & 0 
\end{array}\right)$ & $\begin{array}{c} r \\ b \\ r^c \\ b^c  \end{array}$
\end{tabular}
\end{center}

The eigenvalues of $A^c$ are $\lambda_1 = 1 + \alpha, \lambda_2 = 2p+\alpha-1, \lambda_3 = \lambda_4 =0$. We see that the assumptions (A1)--(A8) hold. The matrix $A^c$ is diagonalizable when $\alpha \neq 1-2p$, and a dual basis for the eigenspaces of $A$ in this case is given by 
\begin{align*}
\bm{v}_1 &= \frac{1}{2}\left(1, 1, \frac{1-p}{1+\alpha}, \frac{1-p}{1+\alpha}\right),\\
\bm{v}_2 &= \frac{1}{2}\left(1, -1, \frac{p-1}{2p + \alpha - 1}, \frac{1-p}{2p+\alpha -2}\right),\\
\bm{v}_3 &= \frac{1}{(1+\alpha)(2p + \alpha -1)}(0,0,1,0),\\
\bm{v}_4 &= \frac{1}{(1+\alpha)(2p + \alpha -1)}(0,0,0,1),\\
\bm{u}_1 &= (1,1,0,0), \\
\bm{u}_2 &= (1,-1,0,0),\\
\bm{u}_3 &= \left( (1-p)^2, (p-1)(\alpha + p), (1+\alpha)(2p+\alpha -1), 0\right)\\
\bm{u}_4 &= \left( (p-1)(\alpha+p), (1-p)^2, 0, (1+\alpha)(2p+\alpha -1)\right).
\end{align*}

 We can therefore apply Theorem \ref{thm:urns}. Using {\sc Mathematica}, the covariance matrices for the limiting distribution for this urn when $1+\alpha = \lambda_1 > 2\lambda_2 = 4p-2 + 2\alpha$ and $1+\alpha = \lambda_1 = 2\lambda_2 = 4p-2 + 2\alpha$ are calculated (see Appendix \ref{App:clusters}). If $\alpha = 1-2p$ (and so $A^c$ is not diagonalizable), then $1 + \alpha =\lambda_1> 2\lambda_2 = 4p - 2 + 2\alpha$ (since $\alpha > -1$), and the calculation for $\Sigma_1^\dagger$ from \eqref{eq:Sigma1special} yields the same result as the calculation for $\Sigma_I$. When $1+\alpha = \lambda_1 < 2\lambda_2 = 4p-2 + 2\alpha$, we conclude from Theorem \ref{thm:urns} (iv) that $n^{(2p+\alpha-1)/(1+\alpha)}(R_n^w, B_n^w, R_n^c, B_n^c)$ converges to $\widehat{Z}\bm{v}_2$, for some random variable $\widehat{Z}$. If we restrict to $R_n^w$ and $B_n^w$, we see that $\widehat{Z}$ is the same random variable $BZ$ as in \eqref{eq:weights}. Restricted to $R_n^c$ and $B_n^c$, the results of Theorem \ref{thm:clusters} follow. 
\end{proof}


A similar urn process to the one above (with balls of activity 0 representing the vertices) can be be used to find limit laws for the number of vertices of each colour. But we can instead use the following observation: if a vertex of one colour contributes to the weight of another vertex of a different colour, then it must be the root of a cluster. Therefore, from the previous proof, we can now derive convergence for the number of vertices of each colour. 

\begin{proof}[Proof of Theorem \ref{thm:vts}]
If we consider again the urn in the previous proof, we can recover the number of vertices $R_n$ and $B_n$ of each colour in our tree. Each red vertex contributes $(1+\alpha)$ to the value $R^w_n$, except those that are roots of red clusters; these contribues 1 to $R^w_n$. The root of a blue cluster contributes $\alpha$ to the weight of its parent, and so $\alpha$ to $R^w_n$. The only root of a blue cluster that does not contribute to $R_n^w$ is the root of $\mathcal{T}_n$ if this root is blue. Using $B$ defined in \eqref{eq:rootcolour}, we see that $R_n^w = (1+\alpha)R_n - \alpha R_n^c + \alpha\left(B_n^c + \frac{B-1}{2}\right)$. Performing the symmetric analysis for $B_n^w$ and rearranging gives 
\begin{align*}
R_n &= \frac{R_n^w + \alpha R_n^c - \alpha\left(\frac{B-1}{2}\right)}{1 + \alpha}, \\
B_n &= \frac{B_n^w + \alpha B_n^c + \alpha \left(\frac{B+1}{2}\right)}{1 + \alpha}.
\end{align*}
When scaled by $\sqrt{n}, \sqrt{\ln n}, $ or $n^{(2p+\alpha-1)(1+\alpha)}$, the last term of each of the equations above vanishes. By the Cram\'er-Wold Theorem, since $R^c_n, B^c_n, R^w_n, B^w_n$ converge jointly in distribution so do linear combinations of these random variables. The limiting distributions are also normal when $\lambda_1 \geq 2 \lambda_2$, and the covariance matrices can be calculated from the covariance matrices in Appendix \ref{App:clusters}. 

As discussed in Remark \ref{rem:vts,p=1/2}, we can treat the special case when $p=1/2$ directly since the number of red vertices is simply given by $R_n = \sum_{i=1}^n X_i$ where $X_i \sim \text{Be}(1/2)$ are independent Bernoulli random variables. Then we can apply the central limit theorem to get 
\[ \frac{R_n - \frac{n}{2}}{\sqrt{n}} \xrightarrow{d} \mathcal{N}\left( 0, \frac{1}{4}\right).\]
A multivariate normal limit law for the number of red and blue vertices follows since $B_n = n - R_n$ and 
\[ a\left(\frac{R_n - \frac{n}{2}}{\sqrt{n}}\right) + b\left(\frac{B_n - \frac{n}{2}}{\sqrt{n}}\right) = (a-b)\left(\frac{R_n - \frac{n}{2}}{\sqrt{n}}\right)\]
converges in distribution to a normal distribution for all $a,b \in \mathbb{R}$, so the Cram\'er-Wold theorem applies. Finally, a quick calculation shows that $\text{Cov}(R_n, B_n) = -\text{Var}(R_n)$, implying the convergence in Theorem \ref{thm:vts}(ii) when $p=1/2$. 
\end{proof}

We turn now to the number of leaves of each colour. 


\begin{proof}[Proof of Theorem \ref{thm:leaves}]
Consider an urn with four colours of balls: $r^l, b^l, r^u, b^u$, each with activity 1. Let $R_n^l, B_n^l, R_n^u, B_n^u$ be the total numbers of the balls of colour $r^l, b^l, r^u, b^u$ respectively at time $n$. The balls of colour $r^l$ and $b^l$ represent red and blue leaves respectively. The other balls represent the remaining weights of the red and blue vertices respectively. 

If a red leaf is chosen at step $n$, this corresponds to choosing a ball of colour $r^l$. Then with probability $1-p$ it is removed and replaced with one ball of colour $b^l$ for the new blue leaf that is added, and $1 + \alpha$ balls of colour $r^u$, representing the weight of the now non-leaf vertex that was chosen. With probability $p$, the ball is placed back in the urn for the new red leaf that was added, along with $1+\alpha$ balls of colour $r^u$, representing the weight of the now non-leaf vertex that was chosen. Therefore, we have 
\[ \mathbb{E}[\xi_{r^l}] = (p-1, 1-p, \alpha + 1, 0)^T,\]
the first column of the intensity matrix $A^l$ for this urn. If a red vertex that is not a leaf is chosen, then an additional $\alpha$ balls of colour $r^u$ are added (for the increase in weight of that vertex), along with either one ball of colour $r^l$ with probability $p$, or one ball of colour $b^l$ with probability $1-p$. Therefore, we have 
\[ \mathbb{E}[\xi_{r^u}] = (p, 1-p, \alpha, 0)^T,\]
the third column of $A^l$. The symmetric arguments hold when balls of colour $b^l$ or $b^u$ are chosen. Therefore, the intensity matrix for this urn is 
\begin{equation}\label{eq:matrixleaves}
\begin{array}{ccc}
& \begin{array}{cccc} r^l \,\,\,\, & \,\,\,\, b^l \,\,\,\, & \,\,\,\, r^u \,\,\,\, & \,\,\,\, b^u \end{array} & \\
A^l= & \left(\begin{array}{cccc}
p-1 & 1-p & p & 1-p \\
1-p & p-1 & 1-p & p \\
\alpha+1 & 0 & \alpha & 0 \\
0 & \alpha+1 & 0 & \alpha
\end{array}\right) & 
\begin{array}{c}r^l \\ b^l \\ r^u \\ b^u \end{array}
\end{array}
\end{equation}
We see immediately that assumptions (A1) -- (A8) hold. The eigenvalues of $A$ are $\lambda_1 = 1+\alpha, \lambda_2 = 2p-1 + \alpha, \lambda_3 = \lambda_4 =-1.$ The matrix $A^l$ is diagonalizable when $\alpha \neq -2p$, and a dual basis for the eigenspaces of $A$ in this case is given by
\begin{align*}
\bm{v}_1 &= \frac{1}{4 + 2\alpha}(1, 1, 1+\alpha, 1+\alpha), \\
\bm{v}_2 &= \frac{1}{2\alpha + 4p}(2p-1, 1-2p, 1+\alpha, -(1 + \alpha)), \\
\bm{v}_3 &= \frac{1}{(2+\alpha)(\alpha +2p)}(1,0,-1,0), \\
\bm{v}_4 &= \frac{1}{(2+\alpha)(\alpha + 2p)}(0,1,0,-1), \\
\bm{u}_1 &= (1,1,1,1), \\
\bm{u}_2 &= (1,-1,1,-1) ,\\
\bm{u}_3 &= \left((1 + \alpha)(1 + \alpha + p), (1 + \alpha)(p-1), 1 - (3 + \alpha)p, (1 + \alpha)(p-1)\right) ,\\
\bm{u}_4 &= \left((1 + \alpha)(p-1), (1 + \alpha)(1 + \alpha + p), (1 + \alpha)(p-1), 1 - (3 + \alpha)p\right). 
\end{align*}
We can therefore apply Theorem \ref{thm:urns}. Using {\sc Mathematica}, the covariance matrix for the limiting distribution for this urn when $1+\alpha = \lambda_1 > 2\lambda_2 = 4p-2 + 2\alpha$ and $1+\alpha = \lambda_1 = 2\lambda_2 = 4p-2 + 2\alpha$ are calculated (see Appendix \ref{App:leaves}).  If $\alpha = -2p$, then $1 + \alpha =\lambda_1 > 2\lambda_2 = 4p - 2 + 2\alpha$ (since $\alpha > -1$), and the calculation for $\Sigma_1^\dagger$ from \eqref{eq:Sigma1special} yields the same result as the calculation for $\Sigma_I$. When $1+\alpha = \lambda_1 < 2\lambda_2 = 4p-2 + 2\alpha$, the limiting distribution depends on the colour of the root vertex, so just as in \eqref{eq:weights}, we multiply by the random variable $B$ defined in \eqref{eq:rootcolour}. Notice also that $R_n^l + R_n^u = R_n^w$ and $B_n^l + B_n^u = B_n^w$, and so from the Cram\'er-Wold Theorem and the uniqueness of limits in distribution, the random variable $Z$ achieved from Theorem \ref{thm:urns}(iv) is identical to the random variable $Z$ in~\eqref{eq:weights}. 

When $p=1/2$, the colour of a newly added vertex does not depend on the colour of its parent. In this case, consider an urn with three colours of balls: $r^l, b^l, v^u$, each with activity 1. The balls of colour $r^l$ and $b^l$ represent red and blue leaves respectively, while $v^u$ represents the remaining weights of all non-leaf vertices. Performing a similar analysis as above, we get the following intensity matrix for this urn:
\begin{equation}\label{eq:matrixleavesspecial}
 \begin{array}{ccc}
& \begin{array}{ccc} \,\,\,\,\, r^l \,\,\,\,\, & \,\,\, b^l \,\,\,\,\, & \,\,v^u\end{array} & \\
A^l = & \left( \begin{array}{ccc}
-\frac{1}{2} & \frac{1}{2} & \frac{1}{2} \\
\frac{1}{2} & -\frac{1}{2} & \frac{1}{2} \\
\alpha + 1 & \alpha + 1 & \alpha
\end{array}\right) & \begin{array}{c} r^l \\ b^l \\ v^u \end{array}
\end{array}
\end{equation}

The eigenvalues of $A$ are $\lambda_1 = 1 + \alpha$, $\lambda_2 = \lambda_3 = -1$, and the matrix is diagonalizable for all valid values of $\alpha$.
A dual basis for the eigenspaces of $A$ is given by 
\begin{align*}
\bm{v}_1 &= \frac{1}{4 + 2\alpha}(1, 1, 2 + 2\alpha), \\
\bm{v}_2 &= \frac{1}{4 + 2\alpha}(1, -3-2\alpha, 2 + 2\alpha),\\
\bm{v}_3 &= \frac{1}{4 + 2\alpha}(1, 1, -2),\\
\bm{u}_1 &= (1,1,1), \\
\bm{u}_2 &= (1,-1,0),\\
\bm{u}_3 &= (2(1 + \alpha), 0, -1).
\end{align*}
Once more, assumptions (A1) -- (A8) hold. By looking at the eigenvalues of $A$, we see immediately that Theorem \ref{thm:urns} (ii) applies. The covariance matrix for this case is included in Appendix \ref{App:leaves}. 

Restricted to $R_n^l$ and $B_n^l$, the results of Theorem \ref{thm:leaves} follow. 
\end{proof}


The proof of Theorem \ref{thm:fringe} follows much the same way as the proof of \cite[Theorem 3.9]{HOJS:17}. Consider a partial ordering $\preccurlyeq$ on the set of all pairs $(T, \varsigma)$, where $T$ is a rooted tree and $\varsigma$ is a two-colouring of the vertices, such that $(T_1, \varsigma_1) \preccurlyeq (T_2, \varsigma_2)$ if $T_1$ is a subtree of $T_2$ (preserving the root) and $\left.\varsigma_2\right|_{T_1} = \varsigma_1$. Let $S = \{(T_1, \varsigma_1), \ldots, (T_q, \varsigma_q)\}$ such that if $(T,\varsigma) \in S$ and $(T', \varsigma') \preccurlyeq (T,\varsigma)$, then $(T', \varsigma') \in S$. Assume that the pairs $(T_1, \varsigma_1), \ldots, (T_q, \varsigma_q)$ are indexed so that if $(T_i, \varsigma_i) \preccurlyeq (T_j, \varsigma_j)$ then $i < j$, and assume that $(T_1, \varsigma_1)$ corresponds to a single red vertex, and $(T_2, \varsigma_2)$ corresponds to a single blue vertex. We define an urn such that
for the tree $\mathcal{T}_{n}$ with colouring $\sigma_n$, if a vertex $v$ is the root of a fringe subtree $T$ isomorphic to $T_i$ with $\left.\sigma_n\right|_{T} = \varsigma_i$ for which $(T_i, \varsigma_i) \in S$ and if $v$ does not belong to another fringe subtree $T'$ isomorphic to $T_j$ with $\left. \sigma_n\right|_{T'} = \varsigma_j$ such that $(T_i, \varsigma_i)\preccurlyeq (T_j, \varsigma_j) \in S$, then $v$ is represented in the urn by the ball of type $i$. 
If $v$ is not the root of a fringe subtree isomorphic to a tree with colouring in $\mathcal{S}$, then $v$ is represented by $\alpha\deg^+(v) + 1$ balls of special type $\ast_r$ if $v$ is red, and $\ast_b$ if $v$ is blue. Let $Y_n^i$ be the number of balls of type $i$ at time $n$, and let $Y_n^{\ast_r}$ and $Y_n^{\ast_b}$ be the number of balls of special type $\ast_r$ and $\ast_b$ respectively at time $n$, and let $\bm{Y}_n = (Y_n^1, \ldots, Y_n^q, Y_n^{\ast_r}, Y_n^{\ast_b})$.

For example, consider $S = \{(T_1, \varsigma_1), \ldots, (T_6, \varsigma_6)\}$, where $(T_i, \varsigma_i)$ are identified on the right side of Figure \ref{fig:fringesubtrees}. A tree $\mathcal{T}_{23}$ with colouring $\sigma_{23}$ is given in Figure \ref{fig:fringesubtrees}. Then the urn we consider will contain two balls of type 1, two balls of type 2, one ball of type 3, one ball of type 4, one ball of type 5, and two balls of type 6. There are a further $7\alpha + 4$ balls of type $\ast_r$ for the remaining red vertices, and $6\alpha + 2$ balls of type $\ast_b$ for the remaining blue vertices. Note that only two red leaves contribute balls of type 1, since the remaining red leaves are subtrees of fringe subtrees isomorphic to $(T_4, \sigma_4)$ or $(T_6, \sigma_6)$. 

\begin{figure}[h!]
\centering
\begin{tikzpicture}[
redv/.style={circle, inner sep=0pt, draw=black, fill=red, thick, minimum size=8pt},
bluv/.style={rectangle, inner sep=0pt, draw=black, fill=cyan, thick, minimum size=8pt}
]

\draw (5,4) -- (4,3) -- (3,2) -- (2,1) ;
\draw (5,4) -- (6,3) -- (7,2) -- (8,1);
\draw (5,4) -- (5,3) -- (4,2);
\draw (6,3) -- (5,2) -- (5,1);
\draw (6,3) -- (6,2);
\draw (3,2) -- (3,1) -- (2,0);
\draw (3,2) -- (4,1) -- (4,0);
\draw (7,2) -- (6,1) -- (6,0);
\draw (7,2) -- (7,1) -- (7,0);
\draw (3,1) -- (3,0);
\draw (4,1) -- (5,0);
\draw (7,1) -- (8,0);

\node[redv] at (5,4){};

\node[redv] at (4,3){};
\node[redv] at (5,3){};
\node[redv] at (6,3){};

\node[bluv] at (3,2){};
\node[bluv] at (4,2){};
\node[redv] at (5,2){};
\node[redv] at (6,2){};
\node[bluv] at (7,2){};

\node[bluv] at (2,1){};
\node[bluv] at (3,1){};
\node[redv] at (4,1){};
\node[redv] at (5,1){};
\node[bluv] at (6,1){};
\node[bluv] at (7,1){};
\node[bluv] at (8,1){};

\node[redv] at (2,0){};
\node[redv] at (3,0){};
\node[bluv] at (4,0){};
\node[bluv] at (5,0){};
\node[redv] at (6,0){};
\node[redv] at (7,0){};
\node[redv] at (8,0){};

\node[circle,draw, dotted, very thick,minimum size = 18pt]  at (2,1){};
\node[circle,draw, dotted, very thick,minimum size = 18pt]  at (6,2){};
\node[circle,draw, dotted, very thick,minimum size = 18pt]  at (5,1){};
\node[circle,draw, dotted, very thick,minimum size = 18pt]  at (8,1){};
\draw[dotted, very thick] (5, 3.35) -- (3.65, 2) -- (4, 1.65) -- (5.35, 3) -- (5, 3.35);
\draw[dotted, very thick] (3, 1.35) -- (1.65, 0) -- (2, -0.35) -- (3.35, -0.35) -- (3.35, 1) -- (3, 1.35);
\draw[dotted, very thick] (4, 1.35) -- (3.65, 1) -- (3.65, -0.35) -- (5.35, -0.35) -- (5.35, 0) -- (4, 1.35);
\draw[dotted, very thick] (6.35, 1.35) -- (5.65, 1.35) -- (5.65, -0.35) -- (6.35, -0.35) -- (6.35, 1.35);
\draw[dotted, very thick] (7, 1.35) -- (6.65, 1) -- (6.65, -0.35) -- (8.35, -0.35) -- (8.35, 0) --(7, 1.35);

\node at (11.75,4.25){\small $(T_1, \varsigma_1) :$};
\node[redv, minimum size = 6pt] at (13,4.25){};

\node at (11.75,3.75){\small $(T_2, \varsigma_2) :$};
\node[bluv, minimum size = 6pt] at (13,3.75){};

\node at (11.75,3){\small $(T_3, \varsigma_3) :$};
\draw (12.75, 3.25) -- (13.25,2.75){};
\node[redv, minimum size = 6pt] at (12.75,3.25){};
\node[bluv, minimum size = 6pt] at (13.25,2.75){};

\node at (11.75,2){\small $(T_4, \varsigma_4) :$};
\draw (12.75, 1.75) -- (13.25,2.25){};
\node[redv, minimum size = 6pt] at (12.75,1.75){};
\node[bluv, minimum size = 6pt] at (13.25,2.25){};

\node at (11.75,1){\small $(T_5, \varsigma_5) :$};
\draw (12.75, 0.75) -- (13, 1.25) -- (13.25, 0.75){};
\node[bluv, minimum size = 6pt] at (12.75,0.75){};
\node[redv, minimum size = 6pt] at (13,1.25){};
\node[bluv, minimum size = 6pt] at (13.25,0.75){};

\node at (11.75,0){\small $(T_6, \varsigma_6) :$};
\draw (12.75, -0.25) -- (13, 0.25) -- (13.25, -0.25){};
\node[redv, minimum size = 6pt] at (12.75,-.25){};
\node[bluv, minimum size = 6pt] at (13,.25){};
\node[redv, minimum size = 6pt] at (13.25,-.25){};
\end{tikzpicture}
\caption{A tree $\mathcal{T}_{23}$ with broadcasting induced colouring $\sigma_{23}$ with fringe subtrees identified.}
\label{fig:fringesubtrees}
\end{figure}
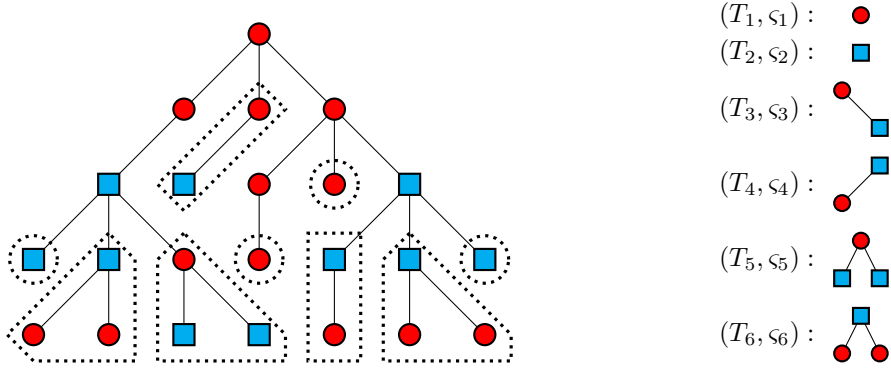

The activity of each ball of type $i$ is given by the sum of the weights of the vertices in the tree $T_i$, which is $a_i := |T_i|(\alpha + 1) - \alpha$. The activities of the balls of special type is 1. When a ball of type $i$ is picked, this corresponds to adding a child $u$ to a vertex $v$ that lies in a fringe subtree isomorphic to $T_i$. Let $(T_j, \varsigma_j)$ denote the fringe subtree with $u$ attached and coloured. If $(T_j, \varsigma_j) \in S$, then the ball of type $i$ is removed and replaced with a ball of type $j$. If $(T_j, \varsigma_j) \notin S$, then the ball of type $i$ is removed, the root $\rho_j$ of $T_j$ is now represented by $\alpha\deg^+(\rho_j) + 1$ balls of special type (with the appropriate colours) that are newly added, and the children of $\rho_j$ are roots to $\deg^+(\rho_j)$ newly considered fringe subtrees. If these subtrees (along with their colouring) appear in $S$, then balls representing them are added. Otherwise, balls of special type are added for the root, and the subtrees of that vertex are considered, continuing this process until all vertices are represented by balls in the urn. If a new vertex $u$ added to $\mathcal{T}_{n}$ is the child of a vertex $v$ that is represented by balls of special type in the urn, then $\alpha$ balls of special type with the appropriate colour are added to the urn, representing the increase in the weight of $v$, while either a ball of type 1 or 2 is added as well, representing the new leaf $u$ added to $\mathcal{T}_{n}$. 

 For $4 \leq k \leq q+2$, let $S_k = \{(T_1, \varsigma_1), \ldots, (T_{k-2}, \varsigma_{k-2})\}$, and let $A_k$ be the intensity matrix for the urn with balls of type $1, \ldots, k-2$ along with $\ast_r$ and $\ast_b$. Let $a_i := |T_i|(\alpha + 1) - \alpha$. 

\begin{proof}[Proof of Theorem \ref{thm:fringe}]
We start with convergence of the random vector $\bm{Y}_n$. We would like to know the eigenvalues of the matrix $A_{q+2}$. We proceed by induction on $k$. Let $4 \leq k \leq q+2$ and consider $A_k$. Let $(T_{i(r)}, \varsigma_{i(r)})$ and $(T_{i(b)}, \varsigma_{i(b)})$ be the longest red and blue path respectively in $S_k$. Then $(A_k)_{ii} = -a_i$ for all $i \neq i(r)$ and $i \neq i(b)$, $(A_k)_{i(r),i(r)} = p - a_{i(r)}$, $(A_k)_{i(b),i(b)} = p - a_{i(b)}$, and $(A_k)_{k-1,k-1} = (A_k)_{kk} = \alpha$. Therefore, we see that 
\[ \text{tr}(A_k) = \alpha + 1 - \sum_{j=1}^{k-2} a_j.\]
The base case is $A_4$, which is precisely the intensity matrix $A^l$ in \eqref{eq:matrixleaves}, and has eigenvalues $\lambda_1 = 1 + \alpha, \lambda_2 = 2p-1 + \alpha, \lambda_3 = \lambda_4 = -1$. The induction step is identical to the one in the proofs of \cite[Theorem 6.2, Theorem 8.2]{HOJS:17}, and the eigenvalues of $A_{k+1}$ are inherited from $A_k$. The last eigenvalue is then given by 
\[ \lambda_{k+1} = \text{tr}(A_{k+1}) - \text{tr}(A_k) =-a_k.\]
Therefore, the eigenvalues of $A_{q+2}$ are 
\[ \lambda_1 = 1+\alpha, \lambda_2 = 2p-1 + \alpha, -a_1, -a_2, \ldots, -a_q.\]

All of the types of balls in the urn are dominating types. This follows since there will always eventually be balls of special type. When a ball of special type is chosen, then either a ball of type 1 or 2 is added (corresponding to the new leaf added to the tree). Since for every $(T_i ,\varsigma_i) \in S$, we have either $(T_1, \varsigma_1) \preccurlyeq (T_i,\varsigma_i)$ or $(T_2, \varsigma_2) \preccurlyeq (T_i,\varsigma_i)$, there is a positive probability of a ball of type $i$ appearing. Finally, for any $(T_i, \varsigma_i) \in S$, there is a positive probability that vertices in a fringe subtree isomorphic to $T_i$ are chosen often enough such that the tree eventually decomposes to balls of special type (and other balls of other types). So if we start the urn with a single ball of any type, then there is a positive probability that any other type of ball will eventually appear. 

All the conditions are met for convergence in distribution of the urn process, and we can apply Theorem \ref{thm:urns}. 
For appropriate right eigenvectors $\bm{v}_1$ and $\bm{v}_2$, we get 
\begin{align}
&\frac{\bm{Y}_n}{n} \xrightarrow{a.s.} \bm{v}_1, &\label{eq:Yi} \\
&\frac{\bm{Y}_n - n\bm{v}_1}{\sqrt{n}} \xrightarrow{d} \mathcal{N}(\bm{0},\Sigma_I^g) & \text{ if } p < \frac{3-\alpha}{4}, \label{eq:Yii}\\
&\frac{\bm{Y}_n - n\bm{v}_1}{\sqrt{n\ln{n}}} \xrightarrow{d} \mathcal{N}(\bm{0},\Sigma_{II}^g) & \text{ if } p = \frac{3-\alpha}{4}, \label{eq:Yiii}\\
&\frac{\bm{Y}_n - n\bm{v}_1}{n^{(2p+\alpha-1)/(1+\alpha)}} \xrightarrow{d} BZ\bm{v}_2 & \text{ if } p > \frac{3-\alpha}{4}. \label{eq:Yiv}
\end{align}
Similar to the proofs above, $R_n^w$ and $B_n^w$ are linear combinations of $Y^1_n, \ldots, Y^q_n, Y^{\ast_r}_n, Y^{\ast_b}_n$, and so by the Cram\'er-Wold Theorem (and the appropriate choice of $\bm{v}_2$), the random variable $Z$ is the same as in \eqref{eq:weights}. 

The case $p=1/2$ is treated similarly by looking at $\bm{Y}_n' = (Y_n^1, \ldots, Y_n^q, Y^{\ast}_n)$, where $Y_n^{\ast} = Y_n^{\ast_r} + Y_n^{\ast_b}$ counts all balls $\ast$ of special type. Since the colour of a new vertex is independent of the colour of its parent, when a ball of type $\ast$ is chosen, $\alpha$ balls of special type $\ast$ are added, while either a ball of type 1 or 2 is added with equal probability. For $3 \leq k \leq q+1$, let $A_k'$ be the intensity matrix for the urn with balls of type $1, \ldots, k-1$ along with balls of type $\ast$. Similar arguments as above hold, but in this case, the base case $A_3'$ is the matrix $A^l$ from \eqref{eq:matrixleavesspecial}. Thus, the eigenvalues $A_{q+1}'$ are 
\[ \lambda_1 = 1+\alpha, -a_1, -a_2, \ldots, -a_q.\]
The conditions are once again met for convergence in distribution of the urn process, and by applying Theorem \ref{thm:urns} and for the appropriate right eigenvector $\bm{v}_1'$ of $A_{q+1}'$, we get
\begin{equation}\label{eq:Y'}
\frac{\bm{Y}_n' - n\bm{v}_1'}{\sqrt{n}} \xrightarrow{d} \mathcal{N}(\bm{0}, \Sigma_I^g).
\end{equation}

The random variables $X^1_n, \ldots, X^q_n$ are linear combinations of $Y^1_n, \ldots, Y^q_n$, and so the convergences of Theorem \ref{thm:fringe} hold by \eqref{eq:Yi} -- \eqref{eq:Y'} above and the Cram\'er-Wold Theorem, though we need to replace $\bm{\mu}$ from \eqref{eq:mu} with some vector $\bm{\mu}'$ for now. We can show that $\bm{\mu}'= \bm{\mu}$ by looking at $\mathbb{E}[(X^1_n, \ldots, X^q_n)]/n$. We have just argued that $(X^1_n, \ldots, X^q_n)/n$ converges almost surely to $\bm{\mu}'$, and since no number of fringe trees exceeds the number of vertices, $(X^1_n, \ldots, X^q_n)/n$ is uniformly bounded. Therefore, $(X^1_n, \ldots, X^q_n)/n$ converges in mean to $\bm{\mu}'$, and so $\mathbb{E}[(X^1_n, \ldots, X^q_n)]/n$ converges to $\bm{\mu}'$. From \cite[Remark 3.10]{HOJS:17} we know that the expected number of fringe subtrees $X^{T_i}$ isomorphic to $T_i$ is given by 
\[\mathbb{E}[X^{T_i}_n] = \frac{\mathbb{P}(\mathcal{T}_{\alpha,k_i} \simeq T_i)\frac{1}{1 + \alpha}}{(k_i - 1 + \frac{1}{1 + \alpha})(k_i + \frac{1}{1+\alpha})} n + O(1),\]
where $k_i$ is the number of vertices in $T_i$.
Since the root of $\mathcal{T}_{n}$ is red or blue with equal probability, then by symmetry, the root of a fringe subtree $T$ isomorphic to $T_i$ is red or blue with equal probability. Then by definition of $\sigma_n$, the colouring $\varsigma$ of $T$ follows the same distribution as $\sigma_{T}$. From this, we conclude that 
\begin{align*}
\mathbb{E}[X_n^i] &= \mathbb{P}(\sigma_{T_i} = \varsigma_i)\mathbb{E}[X_n^{T_i}] \\
&= \mathbb{P}(\sigma_{T_i}= \varsigma_i)\frac{\mathbb{P}(\mathcal{T}_{\alpha,k_i} \simeq T_i)\frac{1}{1 + \alpha}}{(k_i - 1 + \frac{1}{1 + \alpha})(k_i + \frac{1}{1+\alpha})} n + O(1) \\
&= \frac{\mathbb{P}((\mathcal{T}_{\alpha,k_i}, \sigma_{k_i}) \simeq (T_i, \varsigma_i))\frac{1}{\alpha + 1}}{(k_i + \frac{1}{\alpha + 1}- 1)(k_i + \frac{1}{\alpha + 1})}n + O(1).
\end{align*}
Therefore, $\mathbb{E}[(X^1_n, \ldots, X^q_n)]/n$ converges to $\bm{\mu}$, and so $\bm{\mu}'= \bm{\mu}$. 
\end{proof}


\section{Proofs of properties of the root cluster}\label{sec:root}

As mentioned in the introduction, convergence in distribution for the size of the root cluster has previously been proven \cite{BABE:15,BAUR:20, BUSI:18, MOHL:15} using random walks and branching processes. Here we use results from analytic combinatorics to get recursions for the moments of the limiting distributions.

We start by a useful description of the trees studied. Since we are only interested in the size of the root cluster and not the colour of this cluster, we can assume without loss of generality that the root is red. In the following, we define $\phi: \mathbb{Z}_{\geq 0} \rightarrow \mathbb{R}_{\geq 0}$ as 
\begin{equation}\label{eq:phi}
 \phi(\delta) = \begin{cases}
1 & \alpha = 0, \\
\frac{\Gamma(\delta + 1/\alpha)}{\Gamma(1/\alpha)} & \alpha > 0, \\
\frac{d!}{(d-\delta)!} & \alpha = -\frac{1}{d}, d \in \mathbb{Z}^+.
\end{cases}
\end{equation}
For a particular tree $T$ on $n$ vertices, define the {\em weight} of $T$ to be 
\begin{equation}\label{eq:weightsTrees}
 w(T) = \prod_{v \in V(T)}\phi(\deg^+(v)).
\end{equation}
Then the probability of producing the tree $T$ is given by 
\[ \mathbb{P}(\mathcal{T}_{n} = T) = \frac{w(T)}{\sum_{T'}w(T')},\]
see for example \cite[Section 1.1.3]{DRMO:09}. The probability of producing $T$ with a broadcasting induced colouring $\sigma_T$ being the 2-colouring $c$ is then given by 
\[ \mathbb{P}((\mathcal{T}_{n}, \sigma_{n}) = (T,c)| \sigma_n(\rho) = r) = \frac{2\mathbb{P}(\sigma_T = c)w(T)}{\sum_{T'}w(T')},\]
the factor of 2 appearing since we condition on the root being coloured red. 
If we define the weight $\omega(T,c) = 2\mathbb{P}(\sigma_T = c)w(T)$, then 
\[  \mathbb{P}((\mathcal{T}_{n}, \sigma_{n}) = (T,c)| \sigma_n(\rho) = r) = \frac{\omega(T,c)}{\sum_{(T', c')}\omega(T', c')}\]
where $(T', c')$ ranges over all rooted trees on $n$ vertices and over all 2-colourings of the vertices such that the root is red. Symmetrically, define $\omega'(T,c)$ to be the weight of $T$ and $c$ where $c$ is conditioned such that the colour of the root is blue.

Let $r_{n,k}$ be the sum of the weights $\omega(T,c)$ over all  trees with $n$ vertices whose root vertex is red and whose root cluster has size $k$, and let $b_n$ be the sum of the weights $\omega'(T,c)$ over all trees on $n$ vertices with a blue root. Equivalently, $b_n$ is the sum of the weights $w(T)$ over all trees $T$ on $n$ vertices. Then notice that
\[  \mathbb{P}((\mathcal{T}_{n}, \sigma_{n}) = (T,c)) = \frac{\omega(T,c)}{\sum_k r_{n,k}} = \frac{\omega(T,c)}{b_n},\]
and the probability that $\mathcal{T}_{n}$ with colouring $\sigma_{n}$ has a root cluster of size $k'$ is given by $r_{n,k'}/\sum_k r_{n,k}$. 
We develop a recursion formula for $r_{n,k}$. 
Take any tree $T$ on $n$ vertices, with colouring $c$ conditioned on the root $\rho$ being red, whose root cluster is of size $k$, and with $\delta$ subtrees rooted at the children of the root $\rho$. Suppose we order the subtrees such that the first $s$ subtrees $T_1,\ldots, T_s$ have red roots, and the remaining subtrees $T_{s+1}, \ldots, T_{\delta}$ have blue roots. Then the weight of $T$ and $c$ can be written as 
\[ \omega(T,c) =  \phi(\delta)\prod_{i=1}^sp \omega(T_i, c_i)\prod_{j=s+1}^\delta(1-p) \omega'(T_j, c_j),\]
where the $c_i$'s and $c_j$'s are the colouring $c$ restricted to the subtrees $T_i$ and $T_j$ respectively. 
If the trees $T_1, \ldots, T_\delta$ are of size $n_1, \ldots, n_\delta$, then $n_1 + \cdots + n_\delta = n-1$. If the trees $T_1, \ldots, T_s$ have root clusters of size $k_1, \ldots, k_s$, then $k_1 + \cdots + k_s = k-1$.
Now sum over all such trees $T$ on $n$ vertices with root clusters of size $k$. The degree $\delta$ of the root can range from 0 to $n-1$. The number $s$ of children with the colour red ranges from 0 to $\delta$. There are $\binom{\delta}{s}$ ways of choosing these $s$ children. There are $\binom{n-1}{n_1,\ldots, n_\delta}$ ways of distributing the remaining $n-1$ vertices to the $\delta$ subtrees $T_1, \ldots, T_\delta$. Finally, to unorder the subtrees we divide by $\delta!$ to get the recursion
\begin{align}
 r_{n,k} &=  \sum_{\delta=0}^{n-1} \sum_{s=0}^\delta \binom{\delta}{s} \frac{\phi(\delta)}{\delta!} \sum_{n_1, \ldots, n_\delta} \binom{n-1}{n_1,\ldots, n_\delta} \sum_{k_1, \ldots, k_s} \prod_{i=1}^spr_{n_i, k_i}\prod_{j=s+1}^\delta (1-p)b_{n_j} \nonumber \\
&= \sum_{\delta=0}^{\infty} \sum_{s=0}^\delta \binom{\delta}{s} \frac{\phi(\delta)}{\delta!} \sum_{n_1, \ldots, n_\delta} \binom{n-1}{n_1,\ldots, n_\delta} \sum_{k_1, \ldots, k_s}p^s \prod_{i=1}^sr_{n_i, k_i}(1-p)^{\delta-s}\prod_{j=s+1}^\delta b_{n_j},\label{eq:recursion}
\end{align}
where $n_1, \ldots, n_\delta$ range over all non-negative $n_1 + \cdots + n_\delta = n-1$, and $k_1, \ldots, k_s$ range over $k_1 + \cdots + k_s = k-1$. Finally we can let $\delta$ range to infinity since $r_{0,k} = b_0 = 0$. \\

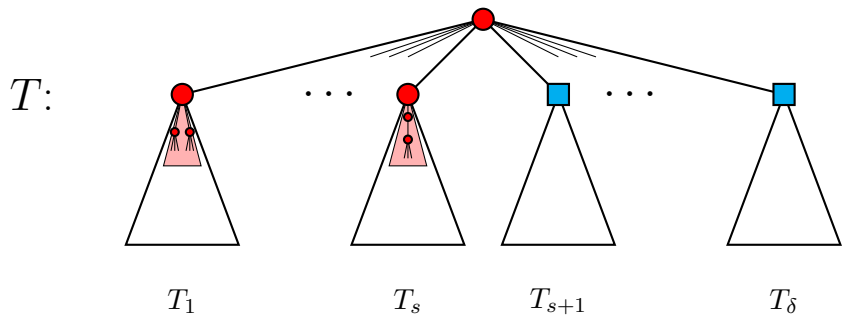
\begin{figure}[h!]
\centering
\begin{tikzpicture}[
redv/.style={circle, inner sep=0pt, draw=black, fill=red, thick, minimum size=8pt},
bluv/.style={rectangle, inner sep=0pt, draw=black, fill=cyan, thick, minimum size=8pt}
]

\filldraw[fill=red!30] (6,9) -- (6.25, 8.05) -- (5.75, 8.05) -- (6,9);

\filldraw[fill=red!30] (9,9) -- (9.25, 8.05) -- (8.75, 8.05) -- (9,9);

\draw (6,9) -- (6.1,8.5);
\draw (6,9) -- (5.9, 8.5);
\draw (5.9, 8.5) -- (5.95, 8.25);
\draw (5.9, 8.5) -- (5.9, 8.25);
\draw (5.9, 8.5) -- (5.85, 8.25);
\draw (6.1, 8.5) -- (6.05, 8.25);
\draw (6.1, 8.5) -- (6.1, 8.25);
\draw (6.1, 8.5) -- (6.15, 8.25);

\draw(9,9) -- (9,8.7) -- (9,8.4) -- (9,8.15);
\draw(9,8.4) -- (9.05, 8.15);
\draw(9,8.4) -- (8.95, 8.15);

\draw[thick] (6,9) -- (6.75,7) -- (5.25,7) -- (6,9);
\draw[thick] (9,9) -- (9.75, 7) -- (8.25, 7) -- (9,9);
\draw[thick] (11,9) -- (11.75, 7) -- (10.25,7) -- (11,9);
\draw[thick] (14,9) -- (14.75,7) -- (13.25,7) -- (14,9);

\draw[thick] (6,9) -- (10,10);
\draw (8.5,9.5) -- (10,10);
\draw (8.75,9.5) -- (10,10);
\draw (9,9.5) -- (10,10);
\draw[thick] (9,9) -- (10,10);
\node at (8, 9){\LARGE $\cdots$};

\draw[thick] (11, 9) -- (10,10);
\draw (11,9.5) -- (10,10);
\draw (11.25,9.5) -- (10,10);
\draw (11.5,9.5) -- (10,10);
\draw[thick] (14,9) -- (10,10);
\node at (12, 9){\LARGE $\cdots$};

\node[redv] at (10,10){};
\node[redv] at (6,9){};
\node[redv] at (9,9){};
\node[bluv] at (11,9){};
\node[bluv] at (14,9){};

\node[redv, minimum size = 3pt] at (6.1, 8.5){};
\node[redv, minimum size = 3pt] at (5.9, 8.5){};

\node[redv, minimum size = 3pt] at (9, 8.7){};
\node[redv, minimum size = 3pt] at (9, 8.4){};

\node at (6,6.25){$T_1$};
\node at (9,6.25){$T_s$};
\node at (11,6.25){$T_{s+1}$};
\node at (14,6.25){$T_{\delta}$};
\node at (4, 9){\LARGE $T$:};

\end{tikzpicture}
\caption{A recursion for the weight of $T$ is established by examining the subtrees rooted at children of the root of $T$.}
\label{fig:recursion}
\end{figure}

We start with the case $\alpha = 0$. It is already known that the size of the root cluster converges to a Mittag-Leffler distribution after proper rescaling. For the sake of completeness and as a simpler example of our more general methods, we reprove the result here. 

Let $R(x,u)$ be the bivariate generating function for $r_{n,k}$, so 
\begin{equation*}
 R(x,u) = \sum_{n,k \geq 1} \frac{r_{n,k}}{n!}x^nu^k,
\end{equation*}
and let $B(x)$ be the exponential generating function for $b_n$. The first thing to notice is that $b_n$ and $\sum_{k=1}^n r_{n,k}$ are simply the number of recursive trees of size $n$, which is $(n-1)!$. Therefore, 
\[ B(x) = R(x,1) =  -\ln(1-x).\]
Using \eqref{eq:recursion}, we establish  a partial differential equation for $R(x,u)$. The resulting differential equation is then solved to get the following closed form for $R(x,u)$. 

\begin{proposition}\label{prop:reca=0}
For $\alpha = 0$, the bivariate generating function $R(x,u)$ is given by
\begin{equation*}
R(x,u) = \frac{-1}{p} \ln(1 - u + u(1-x)^p).
\end{equation*}
\end{proposition}

\begin{proof}
 From \eqref{eq:recursion}, where $\phi(\delta) = 1$ for all $\delta$ (recall \eqref{eq:phi}), we get the partial differential equation 
\begin{align*}
\frac{\partial}{\partial x} R(x,u) &= \sum_{n,k} n \frac{r_{n,k}}{n!} x^{n-1}u^k \\
&= u \sum_{n,k}\sum_{\delta=0}^{n-1}\sum_{s=0}^\delta \binom{\delta}{s}\frac{1}{\delta!}\sum_{n_1, \ldots, n_\delta} \sum_{k_1, \ldots, k_s}   \prod_{i=1}^s \frac{pr_{n_i, k_i}x^{n_i}u^{k_i}}{n_i!} \prod_{j=s+1}^\delta \frac{(1-p)b_{n_j}x^{n_j}}{n_j!} \\
&= u\sum_{\delta=0}^\infty \sum_{s=0}^\delta \frac{1}{\delta!} \binom{\delta}{s} (pR(x,u))^s((1-p)B(x))^{\delta-s} \\
&= u \sum_{\delta=0}^{\infty}\frac{1}{\delta!}(pR(x,u) + (1-p)B(x))^\delta \\
&= u \exp(pR(x,u) + (1-p)B(x)).
\end{align*}
Replacing $B(x)$ with $-\ln(1-x)$ and with the initial condition $R(0,u) = 0$, this linear differential equation has the solution
\[ R(x,u) = \frac{-1}{p} \ln(1 - u + u(1-x)^p).\]
\end{proof}

To prove Theorem \ref{thm:rootclustera=0}, we use the method of moments to establish a limiting distribution for the size of the root cluster after appropriate scaling. 

\begin{proof}[Proof of Theorem \ref{thm:rootclustera=0}]
From Proposition \ref{prop:reca=0}, we calculate
\[ \left.\frac{\partial^k}{\partial u^k} R(x,u)\right|_{u=1} = \frac{1}{p}(k-1)! \big((1-x)^{-p} - 1 \big)^k. \]
We then extract the coefficients, 
\[ [x^n]\left.\frac{\partial^k}{\partial u^k} R(x,u)\right|_{u=1} \sim[x^n] \frac{1}{p}(k-1)! (1-x)^{-pk}
\sim \frac{(k-1)!n^{pk-1}}{p\Gamma(pk)}.\]
Let $\mathcal{C}_n$ be the root cluster at time $n$. The factorial moments of $|\mathcal{C}_n|$ are extracted from the bivariate generating function (see for example \cite[Proposition III.2]{FLSE:09}) to get 
\[ \mathbb{E}\left[|\mathcal{C}_n|(|\mathcal{C}_n| -1) \cdots (|\mathcal{C}_n| - k + 1)\right]= \frac{ [x^n]\left.\frac{\partial^k}{\partial u^k} R(x,u)\right|_{u=1}}{[x^n]R(x,1)} \sim\frac{(k-1)!n^{pk}}{p\Gamma(pk)} .\]
It can be seen (say by induction), that once expanded and scaled by $n^{pk}$, all but the $\mathbb{E}[|\mathcal{C}_n|^k]$ term on the left hand side of the above equation vanish to zero, and thus 
\[ \mathbb{E}\left[\frac{|\mathcal{C}_n|^k}{n^{pk}}\right] \sim \frac{(k-1)!n^{pk}}{n^{pk}p\Gamma(pk)} = \frac{(k-1)!}{p\Gamma(pk)} =\frac{k!}{\Gamma(pk + 1)},\]
which are the moments of the Mittag-Leffler distribution with parameter $p$.
The Mittag-Leffler distribution is uniquely determined by its moments (since its moment generating function, the Mittag-Leffler function $E_p(s) = \sum_{n=0}^\infty \frac{s^n}{\Gamma(pn+1)}$, converges for all values of $s$ \cite{MILE:05}). Therefore, 
\[ \frac{|\mathcal{C}_n|}{n^p} \xrightarrow{d} M_p\]
where $M_p$ has the Mittag-Leffler distribution with parameter $p$. 
\end{proof}

We move on now to the case $\alpha > 0$.  
We again let
\begin{equation*}
R(x,u) = \sum_{n,k}r_{n,k}x^nu^k \quad \text{and} \quad B(x)  =\sum_n b_n x^n.
\end{equation*}
The functions $B(x)$ and $R(x,1)$ are simply the generating function of preferential attachment trees, which is already known (see for example \cite[p.~252]{DRMO:09}) to be 
\begin{equation*}
 B(x) = R(x,1) = 1 - (1 - (1 + 1/\alpha)x)^{\frac{\alpha}{1 + \alpha}}.
\end{equation*}
Unlike the case when $\alpha = 0$, we were unable to derive a closed form for $R(x,u)$. But by applying the method of moments, we only need the $k$'th partial derivatives of $R(x,u)$ with respect to $u$. Define 
\begin{equation*}
R_k(x) := \left. \frac{\partial^k}{\partial u^k}R(x,u)\right|_{u=1},
\end{equation*}
and 
\begin{equation*}
R_0(x) := R(x,1) = 1 - (1 - (1 + 1/\alpha)x)^{\frac{\alpha}{1 + \alpha}}.
\end{equation*}
Throughout the remainder of this section, we will make use of the partial Bell polynomials, which are defined to be 
\[ B_{k,j}(x_1, \ldots, x_{k-j+1}) = \sum_{\substack{m_1 + \cdots + (k-j+1)m_{k-j+1} = k\\ m_1 + \cdots + m_{k-j+1} = j}}k!\prod_{i=1}^{k-j+1} \frac{x_i^{m_i}}{m_i!i!^{m_i}}.\]

\begin{lemma}\label{lem:supera}
Let $\alpha > 0$. Then $R_k(x)$ is analytic on the cut plane \[\mathbb{C}\setminus [1/(1 + 1/\alpha),\infty)\]  and 
\begin{equation*}
R_k(x) = C_k(1 - (1 + 1/\alpha)x)^{\frac{-kp - \alpha(k-1)}{1 + \alpha}} + O\left((1 - (1 + 1/\alpha)x)^{\frac{-kp - \alpha(k-1)}{1 + \alpha}+\varepsilon}\right) 
\end{equation*}
for some $\varepsilon > 0$, where $C_k$ satisfies the recursion $C_1 = \alpha/(p + \alpha)$ and 
\begin{equation*}
(k-1)(p/\alpha +1)C_k = \sum_{j=2}^k \frac{p^j\Gamma(j + 1/\alpha)}{\Gamma(1/\alpha)}B_{k,j}(C_1, \ldots, C_{k-j+1}).
\end{equation*}
\end{lemma}

\begin{proof}
Using the recursion in \eqref{eq:recursion}, where $\phi(d) = \Gamma(d + 1/\alpha)/\Gamma(1/\alpha)$ (recall \eqref{eq:phi}), we get the following partial differential equation:
\begin{align*}
\frac{\partial}{\partial x}R(x,u) &= \sum_{n,k} n \frac{r_{n,k}}{n!}x^{n-1}u^k \\
&= u\sum_{n,k} \sum_{\delta=0}^{n-1}\sum_{s=0}^\delta \binom{\delta}{s}\frac{\phi(\delta)}{\delta!}\sum_{n_1, \ldots , n_\delta}\sum_{k_1, \ldots, k_s}\prod_{i=1}^s \frac{pr_{n_i,k_i}x^{n_i}u^{k_i}}{n_i!} \prod_{j=s+1}^\delta \frac{(1-p)b_{n_j}x^{n_j}}{n_j!} \\
&= u \sum_{\delta=0}^{\infty} \frac{\Gamma(\delta + 1/\alpha)}{\Gamma(1/\alpha)\delta!} \sum_{s=0}^\delta \binom{\delta}{s} (pR(x,u))^s((1-p)B(x))^{\delta-s} \\
&= u \sum_{\delta=0}^{\infty} \frac{\Gamma(\delta + 1/\alpha)}{\Gamma(1/\alpha)\delta!}\left(pR(x,u) + (1-p)B(x)\right)^\delta \\
&= u\left(1 - (pR(x,u) + (1-p)B(x))\right)^{-1/\alpha} \\
&= u\left(1 - \left(pR(x,u) + (1-p)\left(1 - (1 - (1 + 1/\alpha)x)^{\alpha/(1 + \alpha)}\right)\right)\right)^{-1/\alpha}.
\end{align*}

We proceed by strong induction. Using the above differential equation, we see that 
\begin{equation*}
R_1'(x) =\left.\frac{\partial^2}{\partial u \partial x}R(x,u)\right|_{u = 1} = \frac{p}{\alpha - (1 + \alpha)x}R_1(x) + (1 - (1 + 1/\alpha)x)^{\frac{-1}{1 + \alpha}}.
\end{equation*}
Solving this differential equation with the initial condition $R_1(0) = 0$ yields
\begin{equation*}
R_1(x) = \frac{\alpha}{p + \alpha}\left((1 - (1 + 1/\alpha)x)^{\frac{-p}{1 + \alpha} }- (1 - (1 + 1/\alpha)x)^{\frac{\alpha}{1 + \alpha}}\right),
\end{equation*}
which is analytic on the desired cut plane. 

For the inductive step, using the product rule at higher orders of partial differentiation produces 
\begin{align*}
R_k'(x) &= \left.\frac{\partial^{k+1}}{\partial u^k \partial x}R(x,u)\right|_{u = 1} \\
&= \left.\frac{\partial^k}{\partial u^k}u\left(1 - \left(pR(x,u) + (1-p)\left(1 - (1 - (1 + 1/\alpha)x)^{\alpha/(1 + \alpha)}\right)\right)\right)^{-1/\alpha}\right|_{u=1} \\
&=\left. \left( u \frac{\partial^k}{\partial u^k}f(R(x,u)) + k \frac{\partial^{k-1}}{\partial u^{k-1}}f(R(x,u))\right)\right|_{u=1}\\
\end{align*}
where $f(y) = (1 - (py + (1-p)(1-(1-(1 + 1/\alpha)x)^{\frac{\alpha}{1 + \alpha}})))^{-1/\alpha}.$
Define
\[ f^{(m)}(y) := \frac{d^m}{dy^m}f(y). \]
Then 
\begin{equation}\label{eq:f(R0)}
 f^{(m)}(R_0(x)) =  \frac{\Gamma(m + 1/\alpha)}{\Gamma(1/\alpha)}p^m(1 - (1 + 1/\alpha)x)^{\frac{-1 - \alpha m}{1 + \alpha}},
\end{equation}
and by using Fa\'a di Bruno's formula for higher order derivatives (see \cite[p.~139, Theorem C]{COMT:74}), we see that 
\begin{align}
\left.\frac{\partial^k}{\partial u^k}f(R(x,u))\right|_{u=1} &=  \sum_{j=1}^k f^{(j)}(R_0(x))B_{k,j}\left(R_1(x), \ldots, R_{k-j+1}(x)\right) \nonumber\\
&= \frac{p}{\alpha}(1-(1+ 1/\alpha)x)^{-1}R_k(x) + g_k(x), \label{eq:f}
\end{align}
where 
\[ g_k(x) = 
\sum_{j=2}^k f^{(j)}(R_0(x))B_{k,j}\left( R_1(x), \ldots, R_{k-j+1}(x)\right).\]
Since analyticity is preserved under arithmetic operations as well as integration, the analyticity of $R_k(x)$ on the desired cut plane follows by the analyticity in the induction hypothesis. By using the forms of $R_j(x)$ in the induction hypothesis and \eqref{eq:f(R0)}, then for some $\varepsilon >0$, 
\begin{align*}
 g_k(x) &=  G_k(1 - (1+1/\alpha)x)^{\frac{-kp - k\alpha - 1}{1+\alpha}} + O\left( (1 - (1+1/\alpha)x)^{\frac{-kp - k\alpha - 1}{1+\alpha}+ \varepsilon}    \right),
\end{align*}
where 
\[ G_k 
= \sum_{j=2}^k \frac{p^j\Gamma(j + 1/\alpha)}{\Gamma(1/\alpha)} B_{k,j}(C_1, \ldots, C_{k-j+1}).\]
From \eqref{eq:f}, the induction hypothesis, and the assumption $\alpha > 0$, we can also conclude that
\[
 \left.k\frac{\partial^{k-1}}{\partial u^{k-1}}f(R(x,u))\right|_{u=1} = O\left((1 - (1+1/\alpha)x)^{\frac{-kp - k\alpha - 1 + p+\alpha}{1+\alpha}} \right) =O\left( (1 - (1+1/\alpha)x)^{\frac{-kp - k\alpha - 1}{1+\alpha}+ \varepsilon}\right)\]
for some $\varepsilon > 0$. By solving the differential equation
\begin{align*}
R_k'(x) &= \left. \left( u \frac{\partial^k}{\partial u^k}f(R(x,u)) + k \frac{\partial^{k-1}}{\partial u^{k-1}}f(R(x,u))\right)\right|_{u=1} \\
&=  \frac{p}{\alpha}(1 - (1+1/\alpha)x)^{-1}R_k(x) + G_k(1 - (1+1/\alpha)x)^{\frac{-kp - k\alpha - 1}{1+\alpha}} \\
&\hspace{20mm}+ O\left( (1 - (1+1/\alpha)x)^{\frac{-kp - k\alpha - 1}{1+\alpha}+ \varepsilon}    \right), 
\end{align*}
we get that 
\[ R_k(x) = C_k(1 - (1 + 1/\alpha)x)^{\frac{-kp - \alpha(k-1)}{1 + \alpha}} + O\left((1 - (1 + 1/\alpha)x)^{\frac{-kp - \alpha(k-1)}{1 + \alpha}+\varepsilon}\right),\]
where
\[ C_k = \frac{G_k}{(k-1)(p/\alpha + 1)},\]
concluding the proof of the lemma. 
\end{proof}

When proving Theorem \ref{thm:rootclustera>0}, we need to show that our limiting distribution is uniquely determined by its moments. This is accomplished by verifying that the moment generating function exists for some positive radius. To prove this fact, we will instead show that the exponential generating function for the coefficients $C_k$ from the previous lemma 
exists for some positive radius around $x=0$. 

\begin{lemma}\label{lem:Ckpowerseries}
The differential equation 
\begin{equation}\label{eq:diffeqCk}
 xc'(x) = \left(\frac{\alpha}{p + \alpha}\right)\left( c(x) - 1 + \frac{1}{(1-pc(x))^{1/\alpha}}\right)
\end{equation}
has a unique analytic solution for some neighbourhood around $x=0$. Furthermore, this solution can be written as 
\[ c(x) := \sum_{k=1}^\infty \frac{C_k}{k!}x^k.\]
\end{lemma}

\begin{proof}
By using the Taylor expansion we see that 
\[ \frac{1}{(1-px)^{1/\alpha}} = 1 + \frac{p}{\alpha}x + O(x^2)\]
which is analytic on $|x| < 1/p$. Therefore, we can rewrite the differential equation in \eqref{eq:diffeqCk} as 
\begin{align*}
xc'(x) &= \left(\frac{\alpha}{p + \alpha}\right)\left(c(x) - 1 + 1 + \frac{p}{\alpha}c(x) + O\left((c(x))^2\right)\right) \\
&= \left(\frac{\alpha}{p + \alpha}\right)\left(\left(\frac{p + \alpha}{\alpha}\right)c(x) + O\left((c(x))^2\right)\right) \\
&= c(x)f(c(x))
\end{align*}
where $f(x) = 1 + O(x)$, so in particular, $f(0) = 1$. Furthermore, $f(x)$ maintains the same radius of convergence as $\frac{1}{(1-px)^{1/\alpha}}$. We solve the above separable differential equation
\[ \int\frac{dx}{x} =\int \frac{dc}{c} + \int\frac{(1-f(c))dc}{cf(c)}\]
to get 
\[ ce^{F(c)} = Kx\]
for some constant $K$, where $F(y) = \int\frac{(1-f(c))dc}{cf(c)}$. The analyticity of $F(x)$ in some neighbourhood of $x=0$ is guaranteed by preservation of analyticity through integration and the analyticity of $\frac{1-f(c)}{cf(c)}$, which is itself analytic due to the analyticity of $f(x)$ and the fact that $f(0) = 1$. Thus, using the implicit value theorem, there exists a unique analytic function $c(x)$ in the neighbourhood of $x=0$ such that $c(0) = 0$. 

To prove the last part of the lemma, it suffices to show that the power series 
\[ c(x) := \sum_{k=1}^\infty \frac{C_k}{k!}x^k\]
satisfies the differential equation \eqref{eq:diffeqCk}. 
Recall the recursion for $C_k$ given in Lemma \ref{lem:supera}, which states that 
\[ \sum_{j=2}^k p^j \frac{\Gamma(j + 1/\alpha)}{\Gamma(1/\alpha)} B_{k,j}(C_1, \ldots, C_{k-j+1}) = G_k = (k-1)(p/\alpha+1)C_k.\]
Recall that 
\[ \sum_{k=1}^{\infty} \frac{\Gamma(k + 1/\alpha)p^k}{\Gamma(1/\alpha)k!}x^k = \frac{1}{(1 - px)^{1/\alpha}} - 1.\]
Then by using known results about composition of functions and Bell polynomials (see e.g. \cite[p.~137, Theorem A]{COMT:74}), 
\begin{align*}
\frac{1}{(1 - pc(x))^{1/\alpha}} - 1 &= \sum_{k=1}^{\infty} \frac{\sum_{j=1}^k p^j \frac{\Gamma(j+1/\alpha)}{\Gamma(1/\alpha)} B_{k,j}(C_1, \ldots, C_{k-j+1})}{k!}x^k \\
&= \sum_{k=1}^\infty \frac{\alpha G_k + pC_k}{\alpha k!}x^k \\
&= \sum_{k=1}^\infty \frac{k(p + \alpha)C_k}{\alpha k!}x^k - \sum_{k=1}^\infty \frac{C_k}{k!}x^k \\
&= \left(\frac{p + \alpha}{\alpha}\right)xc'(x) - c(x),
\end{align*}
which can be rearranged to give \eqref{eq:diffeqCk}. 
\end{proof}

We now have all the tools necessary to prove Theorem \ref{thm:rootclustera>0}.

\begin{proof}[Proof of Theorem \ref{thm:rootclustera>0}]
Using a transfer theorem (see \cite[Corollary VI.1]{FLSE:09}) and Lemma \ref{lem:supera}, 
\[ \left. [x^n] \frac{\partial^k}{\partial u^k} R(x,u)\right|_{u=1} \sim \frac{C_k(1+1/\alpha)^n n^{\frac{kp+\alpha(k-1)}{1 + \alpha} -1}}{\Gamma((kp+\alpha(k-1))/(1+\alpha))}  \]
and 
\[ [x^n] R(x,1) \sim -\frac{(1 + 1/\alpha)^n n^{-\frac{\alpha}{1+\alpha} - 1}}{\Gamma(-\alpha/(1 + \alpha))}.\]
Let $\mathcal{C}_n$ be the root cluster at time $n$. The factorial moments of $|\mathcal{C}_n|$ are extracted from the bivariate generating function (see for example \cite[Proposition III.2]{FLSE:09}) to get 
\[ \mathbb{E}[|\mathcal{C}_n|(|\mathcal{C}_n| - 1)\cdots (|\mathcal{C}_n| - k + 1)] = \frac{[x^n]R_k(x)}{[x^n]R(x,1)} \sim \frac{C_kn^{\frac{k(p+\alpha)}{1 + \alpha}}(1 + \alpha)\Gamma(1/(1+\alpha))}{\alpha\Gamma((kp+\alpha(k-1))/(\alpha+1))} .\]
It can be seen (say by induction) that once expanded and scaled by $n^{k(p+\alpha)/(1 + \alpha)}$, all but the $\mathbb{E}[|C_n|^k]$ term on the left hand side of the above equation vanish to zero, and thus
\[ \mathbb{E}\left[\frac{|\mathcal{C}_n|^k}{n^{k(p+\alpha)/(1 + \alpha)}}\right] \rightarrow \frac{C_k(1+\alpha) \Gamma(1/(1 + \alpha))}{\alpha \Gamma((kp + \alpha(k-1))/(\alpha + 1))} = M_k.\]

For all $k$ large enough, $M_k < C_k$, and so $m(x) = 1 + \sum_{k=1}^\infty \frac{M_k}{k!}x^k$ has greater or equal radius of convergence as $c(x) = \sum_{k=1}^\infty \frac{C_k}{k!}x^k$, which is guaranteed to be nonzero by Lemma \ref{lem:Ckpowerseries}. Let $\mathcal{C}$ be the distribution uniquely determined by its moments $M_k$. Then by using the method of moments, we have shown that 
\[ \frac{|\mathcal{C}_n|}{n^{(p+\alpha)/(1 + \alpha)}} \xrightarrow{d} \mathcal{C}.\]
\end{proof}

In general, we were unable to derive a closed form for $C_k$. We were, however, able to derive a closed form when $\alpha = 1$.

\begin{proof}[Proof of Theorem \ref{thm:rootclustera=1}]
We use Lemma \ref{lem:Ckpowerseries}, and replace $\alpha$ with 1 to get 
\[c'(x) = \frac{1}{x(p+1)} \left( \frac{1}{1-p c(x)} + c(x) - 1 \right) = \frac{c(x)(1+p-pc(x))}{x(p+1)(1-pc(x))},\]
which is rewritten as
\[\int \frac{dx}{x} = \int \frac{(p+1)(1-pc)}{c(1+p-pc)}\,dc = \int \left( \frac{1}{c} - \frac{p^2}{1+p-pc} \right)\,dc.\]
So
\[\ln x = \ln c + p\ln (1+p-p c) + K.\]
Since we know that $c(x) = \frac{x}{p+1} + O(x^2)$, the constant $K$ is $(1-p)\ln(p+1)$. So
\[\ln c(x) = \ln x - p \ln(1+p-pc(x)) +(p-1)\ln(p+1)\]
or
\[c(x) = \frac{x}{p+1} \left(1 - \frac{p}{p+1} c(x) \right)^{-p}.\]
Applying the Lagrange inversion formula (see for example \cite[Theorem A.2]{FLSE:09}) to this functional equation yields
\[[x^k] c(x) = \frac{1}{k} [t^{k-1}] \frac{1}{(p+1)^k} \left(1 - \frac{pt}{p+1} \right)^{-kp} = \frac{1}{k(p+1)^k} \left( \frac{p}{p+1} \right)^{n-1} \binom{kp+k-2}{k-1}.\]
So finally
\[C_k = k! [x^k] c(x) = \frac{(k-1)! p^{k-1}}{(p+1)^{2k-1}} \binom{kp+k-2}{k-1} = \frac{p^{k-1}\Gamma(kp+k-1)}{(p+1)^{2k-1}\Gamma(kp)}.\]
Theorem \ref{thm:rootclustera=1} now follows from the above derivation and Theorem \ref{thm:rootclustera>0}. 
\end{proof}


We turn our attention to the case $\alpha = -1/d$ for some integer $d \geq 2$, and $\alpha > -p$. In this case the functions $B(x)$ and $R(x,1)$ are equal to the generating function for increasing $d$-ary trees, which is known (see for example \cite[Lemma 6.5]{DRMO:09}) to be
\[ B(x) = R(x,1) = (1 - (d-1)x)^{-\frac{1}{d-1}} - 1.\]
Once more, we were unable to derive a closed form for $R(x,u)$ in this case. Recall the notation
\begin{equation*}
R_k(x) := \left. \frac{\partial^k}{\partial u^k}R(x,u)\right|_{u=1},
\end{equation*}
and 
\begin{equation*}
R_0(x) := R(x,1) = (1 - (d-1)x)^{-\frac{1}{d-1}} - 1.
\end{equation*}

\begin{lemma}\label{lem:superb}
Let $d\geq 2 $ be a positive integer and let $p > 1/d$. Then $R_k(x)$ is analytic on the cut plane 
\[ \mathbb{C}\setminus [1/(d-1), \infty),\]
and 
\[ R_k(x) = D_k(1 - (d-1)x)^{\frac{-kpd + k -1}{d-1}} + O\left((1 - (d-1)x)^{\frac{-kpd + k -1}{d-1} + \varepsilon}\right), \]
for some $\varepsilon > 0$, where $D_k$ satisfies the recursion $D_1 = 1/(pd-1)$ and 
\begin{equation*}
(k-1)(pd-1)D_k = \sum_{j=2}^{\min\{k,d\}} \frac{p^jd!}{(d-j)!} B_{k,j}(D_1, \ldots, D_{k-j+1}).
\end{equation*}
\end{lemma}

Since the proof of Lemma \ref{lem:superb} follows much the same way as the proof of Lemma \ref{lem:supera}, the argument is relegated to Appendix \ref{App:LemmaProofs}. Much like the case above for $\alpha >0$, we will prove the existence of the moment generating function of our limiting distribution in a neighbourhood of $0$, and this is done by studying the exponential generating function of $D_k$. 

\begin{lemma}\label{lem:Dkpowerseries}
The differential equation 
\begin{equation}\label{eq:diffeqDk}
 xt'(x) = \frac{(1+pt(x))^d - t(x) - 1}{pd-1}
\end{equation}
has a unique analytic solution for some neighbourhood around $x=0$. Furthermore, this solution can be written as 
\[ t(x) := \sum_{k=1}^\infty \frac{D_k}{k!}x^k.\]
\end{lemma}

\begin{proof}
By using the Binomial Theorem, we rewrite the differential equation as
\begin{align*}
xt'(x) &= \frac{1 + pdt(x) + \sum_{k=2}^d\binom{d}{k}(pt(x))^k -  t(x) - 1}{pd-1}\\
&= t(x) + \left(\frac{1}{pd-1}\right)\sum_{k=2}^d\binom{d}{k}(pt(x))^k \\
&= t(x)g(t(x)),
\end{align*}
where 
\[ g(x) = 1 + \left(\frac{1}{pd-1}\right)\sum_{k=2}^d\binom{d}{k}p^kx^{k-1},\]
which is simply a polynomial (and so an entire function), and $g(0) = 1$. The remainder of the existence part of the proof now follows much the same as that of Lemma \ref{lem:Ckpowerseries}. 

To prove the last part of the theorem,  recall the recursion of $D_k$ given in Lemma \ref{lem:superb}. Then
\[ \sum_{j=2}^{\min\{k,d\}} p^j\frac{d!}{(d-j)!} B_{k,j}(D_1, \ldots, D_{k-j+1}) = (k-1)(pd-1)D_k.\]
By using known results about composition of functions and Bell polynomials,
\begin{align*}
 (1 + pt(x))^d - 1 &= \sum_{k=1}^\infty \frac{\sum_{j=1}^{\min\{k,d\}} p^j\frac{d!}{(d-j)!} B_{k,j}(D_1, \ldots, D_{k-j+1})}{k!}x^k \\
&= \sum_{k=1}^\infty \frac{(k-1)(pd-1)D_k + pdD_k}{k!}x^k \\
&= \sum_{k=1}^\infty \frac{k(pd-1)D_k}{k!}x^k + \sum_{k=1}^\infty \frac{D_k}{k!}x^k \\
&= (pd-1)xt'(x) + t(x),
\end{align*}
which can be rearranged to give \eqref{eq:diffeqDk}. 
\end{proof}

The proof of  Theorem \ref{thm:rootclusterbsup} now follows in much the same way as the proof of Theorem~\ref{thm:rootclustera>0}.

\begin{proof}[Proof of Theorem \ref{thm:rootclusterbsup}]
Using a transfer theorem (see again \cite[Corollary VI.1]{FLSE:09}) and Lemma \ref{lem:superb}, 
\[ \left. [x^n] \frac{\partial^k}{\partial u^k} R(x,u)\right|_{u=1} \sim \frac{D_k(d-1)^n n^{\frac{kpd - k +1}{d-1} -1}}{\Gamma((kpd - k +1)/(d-1))}  \]
and 
\[ [x^n] R(x,1) \sim \frac{(d-1)^n n^{\frac{1}{d-1} - 1}}{\Gamma(1/(d-1))}.\]
Let $\mathcal{C}_n$ be the root cluster at time $n$. The factorial moments of $|\mathcal{C}_n|$ are extracted from the bivariate generating function (see for example \cite[Proposition III.2]{FLSE:09}), and once scaled by $n^{k(pd-1)/(d-1)}$, we get 
\[ \mathbb{E}\left[\frac{|\mathcal{C}_n|^k}{n^{k(pd-1)/(d-1)}}\right] \rightarrow \frac{D_k \Gamma(1/(d-1))}{\Gamma((kpd - k + 1)/(d-1))} = M_k.\]
For all $k$ large enough, $M_k < D_k$, and so $m(x) = 1 + \sum_{k=1}^\infty \frac{M_k}{k!}x^k$ has a greater or equal radius of convergence as $t(x) = \sum_{k=1}^\infty \frac{D_k}{k!}x^k$, which is guaranteed to be nonzero by Lemma \ref{lem:Dkpowerseries}. Let $\mathcal{C}$ be the distribution uniquely determined by its moments $M_k$. Then by using the method of moments, we have shown that 
\[ \frac{|\mathcal{C}_n|}{n^{(pd-1)/(d-1)}} \xrightarrow{d} \mathcal{C}.\]

\end{proof}

We were unable to find a closed form for $D_k$ in general. However, a closed form can be found in the case of binary search trees, when $d=2$.

\begin{proof}[Proof of Theorem \ref{thm:rootclusterb=2}]
We use Lemma \ref{lem:Dkpowerseries} and replace $d$ with 2 to get
\[ t'(x) = \frac{t(x)(p^2t(x) + 2p - 1)}{(2p-1)x},\]
which is rewritten as 
\[ \int\frac{dx}{x} = \int \frac{2p-1}{t(x)(p^2t(x) + 2p - 1)}dt(x) = \int\left( \frac{1}{t(x)} - \frac{p^2}{p^2t(x)+2p-1}\right)dt(x).\]
So 
\[ \ln x = \ln t(x) - \ln(p^2t(x) + 2p - 1) + K.\]
Since $t(x) = \frac{x}{2p-1} + O(x^2)$, the constant $K$ is $2\ln(2p-1)$, so 
\[t(x) = x\left(\frac{1}{2p-1} + \frac{p^2t(x)}{(2p-1)^2}\right).\]
The Lagrange inversion formula yields
\[ [x^k]t(x) = \frac{1}{k}[y^{k-1}]\frac{1}{(2p-1)^k}\left(1 + \frac{p^2 y}{2p-1}\right)^k = \binom{k}{k-1}\frac{p^{2(k-1)}}{k(2p-1)^{2k-1}}.\]
So
\[ D_k = k![x^k]t(x) = \frac{k!p^{2(k-1)}}{(2p-1)^{2k-1}}.\]
Theorem \ref{thm:rootclusterb=2} now follows from the above derivation and Theorem \ref{thm:rootclusterbsup}. 
\end{proof}


We now look at the cases when the root cluster is finite. Our strategy in these cases is to look at bond percolation on the complete infinite $d$-ary tree $T_d$. The root cluster $\mathcal{K}_d$ after performing bond percolation on $T_d$ is distributed as a Galton-Watson tree with binomial $\text{Bin}(d,p)$ offspring distribution. The  size (total progeny) of such (finite) trees is known to follow 
\[ \mathbb{P}(|\mathcal{K}_d|= k) = \frac{1}{k}\mathbb{P}(X_1 + \cdots + X_k = k-1)\]
where $X_1, \ldots, X_k$ are independent binomial random variables $X_i \sim \text{Bin}(d,p)$ (this result was proved by Otter \cite{OTTE:49}; a more general result was proved by Dwass \cite{DWAS:69}. See also \cite[Exercises 2.2--2.4]{HEHO:17}). Thus
\begin{equation}\label{eq:completecluster}
\mathbb{P}(|\mathcal{K}_d| = k) = \frac{1}{k}\binom{kd}{k-1}p^{k-1}(1-p)^{kd - k +1}.
\end{equation}
If we now let $\mathcal{T}_{n}$ be the rooted subtree of $T_d$ corresponding to a random increasing $d$-ary tree at time $n$, then $\mathcal{C}_n \sim \mathcal{T}_n \cap \mathcal{K}_d$ is distributed as the root cluster of $\mathcal{T}_n$ with a random broadcasting induced colouring $\sigma_n$, where the intersection is the subtree of both $\mathcal{T}_n$ and $\mathcal{K}_d$. For example in Figure \ref{fig:subcritical}, we see a tree $\mathcal{T}_{9}$, with thick edges in the figure, grown on a complete infinite 3-ary tree $T_3$. Bond percolation has been performed on $T_3$ (dashed edges represent edges that were removed), and the root cluster $\mathcal{K}_3$ is shown surrounded by dotted lines. The root cluster $\mathcal{C}_{9}$ of $\mathcal{T}_{9}$ is the intersection of $\mathcal{K}_3$ and $\mathcal{T}_{9}$.

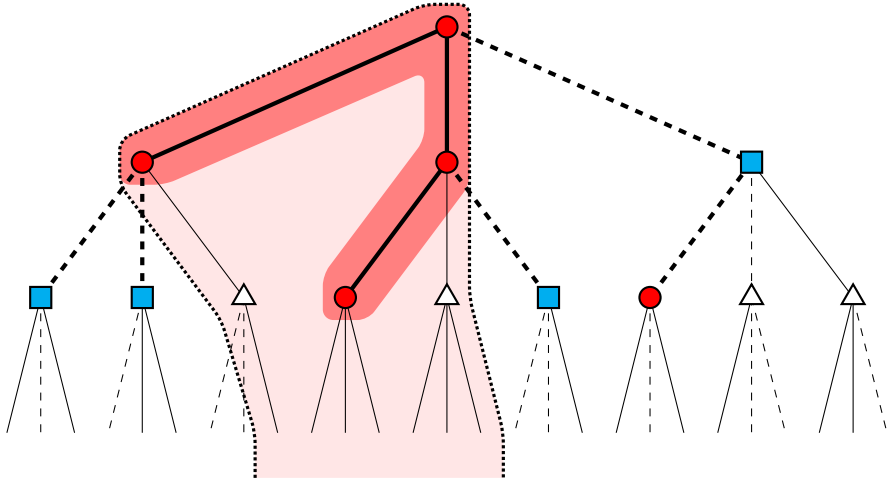
\begin{figure}[h!]
\centering
\begin{tikzpicture}[scale =1.5,
redv/.style={circle, inner sep=0pt, draw=black, fill=red, thick, minimum size=8pt},
bluv/.style={rectangle, inner sep=0pt, draw=black, fill=cyan, thick, minimum size=8pt},
nonv/.style={regular polygon, regular polygon sides=3, inner sep=0pt, draw=black, fill=white, thick, minimum size=10pt}
]

\fill[fill=red!10, rounded corners] (2.4, 2) -- (2.4, 2.4)  -- ( 2.1, 3.4) -- (1.2, 4.6) -- (1.2, 5) -- (3.9, 6.2) -- (4.1, 6.2) -- (4.3, 6.2) -- (4.3, 4.6) -- (4.3, 3.6) -- (4.6, 2.4) -- (4.6, 2);

\fill[fill=red!50, rounded corners]  
 (4.1, 6.2) -- (4.3, 6.2) -- (4.3, 4.6) -- (3.4, 3.4) -- (3, 3.4) -- (3, 3.8) -- (3.9, 5) -- (3.9, 5.6) -- (1.6, 4.6) -- (1.2, 4.6) -- (1.2, 5) -- (3.9, 6.2) -- (4.1, 6.2);

\draw[ very thick, rounded corners, densely dotted] (2.4, 2) -- (2.4, 2.4)  -- ( 2.1, 3.4) -- (1.2, 4.6) -- (1.2, 5) -- (3.9, 6.2) -- (4.1, 6.2) -- (4.3, 6.2) -- (4.3, 4.6) -- (4.3, 3.6) -- (4.6, 2.4) -- (4.6, 2);

\draw[ultra thick] (4.1,6) -- (1.4, 4.8);
\draw[ultra thick] (4.1, 6) -- (4.1, 4.8);
\draw[ultra thick, dashed] (4.1, 6) -- (6.8, 4.8);

\draw[ultra thick, dashed] (1.4, 4.8) -- (.5, 3.6);
\draw[ultra thick, dashed] (1.4, 4.8) -- (1.4, 3.6);
\draw[] (1.4, 4.8) -- (2.3, 3.6);
\draw[ultra thick] (4.1, 4.8) -- (3.2, 3.6);
\draw (4.1, 4.8) -- (4.1, 3.6);
\draw[ultra thick, dashed] (4.1, 4.8) -- (5, 3.6);
\draw[ultra thick, dashed] (6.8, 4.8) -- (5.9, 3.6);
\draw[dashed] (6.8, 4.8) -- (6.8, 3.6);
\draw (6.8, 4.8) -- (7.7, 3.6);

\foreach \x in {0, ..., 26}{
	\ifthenelse{\x = 1 \OR \x=3 \OR \x=6 \OR \x=7 \OR \x=15 \OR \x=16\OR \x=19 \OR \x=22 \OR \x=23 \OR \x=26}{
		\draw[ dashed] ({.5 + .9*floor(\x/3)}, 3.6) -- (.2+.3*\x, 2.4);}{
		\draw ({.5 + .9*floor(\x/3)}, 3.6) -- (.2+.3*\x, 2.4);}}

\node[redv] at (4.1, 6){};

\node[redv] at (1.4, 4.8){};
\node[redv] at (4.1, 4.8){};
\node[bluv] at (6.8, 4.8){};
\foreach \x in {0,...,8}{
	\ifthenelse{\x=0 \OR \x=1 \OR \x=3  \OR \x=5 \OR \x=6}{
	\ifthenelse{\x=0 \OR \x=1 \OR \x=5}{
		\node[bluv] at (.5+.9*\x, 3.6){};}{
		\node[redv] at (.5+.9*\x, 3.6){};}}{
		\node[nonv] at (.5+.9*\x, 3.6){};}}

\end{tikzpicture}
\caption{A random increasing 3-ary tree is grown on a complete infinite 3-ary tree with bond percolation performed. The root cluster $\mathcal{C}_{9}$ has 4 vertices at this stage. }
\label{fig:subcritical}
\end{figure}

\begin{proof}[Proof of Theorem \ref{thm:rootclusterba.s.}]
First note that with the set-up above, $|\mathcal{C}_n| \leq |\mathcal{K}_d|$. 
We look at the {\em fill-up} or {\em saturation} level $\overline{H}_n$ of $\mathcal{T}_n$, the greatest value $m$ for which $\mathcal{T}_n$ has $d^m$ vertices at distance $m$ from the root.  From \cite[Theorem 6.46]{DRMO:09}, we know that 
\begin{equation}\label{eq:filluplevel}
 \mathbb{E}[\overline{H}_n] \sim a\ln{n}, \hspace{10mm} \mathbb{P}\left(|\overline{H}_n -\mathbb{E}[\overline{H}_n]| \geq \eta \right) = O(e^{-c\eta}) 
\end{equation}
for some constants $a,c > 0$ (not depending on $n$). 

First consider the event $|\mathcal{K}_d| = \infty$. Then some vertex of $\mathcal{K}_d$ is present at every level of $T_d$.  From this, we get the rough bound $|\mathcal{C}_n| \geq \overline{H}_n$, which grows to infinity almost surely thanks to \eqref{eq:filluplevel}. 

Now condition on the event $|\mathcal{K}_d| < \infty$. First we prove convergence in probability. Then almost sure convergence follows from convergence in probability and the fact that $\{|\mathcal{C}_n|\}_{n=1}^\infty$ are increasing random variables (see for example \cite[ch. 5, Theorem 3.5]{GUT:13}). Letting $h(\mathcal{K}_d)$ be the height of $\mathcal{K}_d$, then it is immediate that if $\overline{H}_n \geq h(\mathcal{K}_d)$, then $\mathcal{C}_n = \mathcal{K}_d$. Therefore, 
\begin{align*}
\mathbb{P}(|\mathcal{C}_n - \mathcal{K}_d| >0) &\leq \mathbb{P}(\overline{H}_n < h(\mathcal{K}_d)) \\
&\leq \mathbb{P}\left(\overline{H}_n < \frac{a}{2}\ln{n}\right) + \mathbb{P}\left(h(\mathcal{K}_d) > \frac{a}{2}\ln{n}\right).
\end{align*}
The last line tends to zero thanks to \eqref{eq:filluplevel} and the fact that $\mathcal{K}_d$ is finite. The desired convergence in probability is then achieved, completing the proof.
\end{proof}

\begin{proof}[Proof of Corollary \ref{thm:rootclusterbsub}]
This follows immediately from Theorem \ref{thm:rootclusterba.s.} and \eqref{eq:completecluster}.
\end{proof}

When $p < 1/d$, the distribution described by \eqref{eq:completecluster} has finite moments. Since $\{|\mathcal{C}_n|\}_{n=1}^\infty$ consists of increasing positive random variables bounded by $\mathcal{K}_d$, then their moments are uniformly bounded by the moments of $\mathcal{K}_d$. Thus, along with the almost sure convergence of Theorem \ref{thm:rootclusterba.s.}, convergence in all moments holds as well (see \cite[ch. 5, Theorem 5.2]{GUT:13}).

Howewer, when $p= 1/d$, the distribution described by \eqref{eq:completecluster} does not even have finite expectation. 
We can, however, derive asymptotic results for the moments of $|\mathcal{C}_n|$. 


We start by once more approximating the functions $R_k(x)$.

\begin{lemma}\label{lem:crit}
Let $d\geq 2 $ be a positive integer and let $p = 1/d$. Then $R_k(x)$ is analytic on the cut plane 
\[ \mathbb{C}\setminus [1/(d-1), \infty),\]
and 
\begin{align*}
 R_k(x) &= -E_k(1-(d-1)x)^{\frac{-1}{d-1}}\ln^{2k-1}(1-(d-1)x) \\
&\hspace{20mm} + O\left((1-(d-1)x)^{\frac{-1}{d-1}}\ln^{2k-2}(1-(d-1)x)\right).
\end{align*}
where $E_k$ satisfies the recursion $E_1 = 1/(d-1)$ and 
\begin{equation*}
(2k-1)E_k = \frac{1}{2d}\sum_{j=1}^{k-1}\binom{k}{j}E_j E_{k-j}.
\end{equation*}
\end{lemma}

Since the proof of Lemma \ref{lem:crit} also follows much the same way as the proof of Lemma \ref{lem:supera}, the argument is relegated to Appendix \ref{App:LemmaProofs}.

\begin{proof}[Proof of Theorem \ref{thm:rootclusterbcrit}]
From the approximations of the functions $R_k(x)$ in the previous proofs, we conclude by a transfer theorem \cite[Corollary VI.1]{FLSE:09} (or also \cite[Th\'eor\`eme A]{JUNG:31}) that 
\[\left. [x^n] \frac{\partial^k}{\partial u^k} R(x,u)\right|_{u=1} \sim \frac{E_k(d-1)^nn^{\frac{1}{d-1} - 1}}{\Gamma(1/(d-1))}\ln^{2k-1}n,\]
and 
\[ [x^n]R(x,1) \sim \frac{(d-1)^n n^{\frac{1}{d-1} - 1}}{\Gamma(1/(d-1))}.\]
Therefore, we see that 
\[ \mathbb{E}[|\mathcal{C}_n|^k] = \frac{[x^n]R_k(x)}{[x^n]R(x,1)} \sim E_k\ln^{2k-1}n.\]
\end{proof}

\section*{Acknowledgment}
We would like to thank Svante Janson and Gabriel Berzunza Ojeda for helpful discussions related to this work. We would further like to thank Gabriel Berzunza Ojeda for pointing us towards the references on percolation in random recursive trees.


\newpage
\appendix
\section{Covariance matrices}

\subsection{Number of clusters of each colour}\label{App:clusters}

For $p < (3 - \alpha)/4$:

\[ \Sigma_I^c =\frac{1}{ 4(\alpha - 3 + 4p)} \left(
\begin{array}{cccc}
\sigma_{1,1}^c & \sigma_{1,2}^c & \sigma_{1,3}^c & \sigma_{1,4}^c \\
\sigma_{2,1}^c & \sigma_{2,2}^c & \sigma_{2,3}^c & \sigma_{2,4}^c \\
\sigma_{3,1}^c & \sigma_{3,2}^c & \sigma_{3,3}^c & \sigma_{3,4}^c \\
\sigma_{4,1}^c & \sigma_{4,2}^c & \sigma_{4,3}^c & \sigma_{4,4}^c \\
\end{array}
\right)\]
where 
\begin{align*}
\sigma_{1,1}^c = \sigma_{2,2}^c &= -(\alpha+1) \left(\alpha^2+(4 p-2) \alpha+1\right)  \\
\sigma_{2,1}^c = \sigma_{1,2}^c &= (\alpha+1)\left(\alpha^2+(4 p-2) \alpha+1\right)  \\
\sigma_{3,3}^c = \sigma_{4,4}^c &= -(p-1)^2 (\alpha+4p+1) \\
\sigma_{3,4}^c = \sigma_{4,3}^c &= -(p-1) \left(4 p^2+\alpha p-3 p+\alpha+1\right) \\
\sigma_{3,1}^c = \sigma_{4,2}^c &=\sigma_{1,3}^c = \sigma_{2,4}^c = (\alpha+1) (\alpha+2 p-1)\\
\sigma_{3,2}^c = \sigma_{4,1}^c &=\sigma_{2,3}^c = \sigma_{1,4}^c =-(\alpha+1) (\alpha+2 p-1) 
\end{align*}

\noindent For $p = (3-\alpha)/4$:

\[\Sigma_{II}^c  = \frac{1}{4}\left(
\begin{array}{cccc}
\alpha + 1 & -\alpha - 1 & \frac{-\alpha - 1}{2} & \frac{\alpha+1}{2} \\
 -\alpha - 1 & \alpha + 1& \frac{\alpha +1}{2} & \frac{-\alpha - 1}{2}  \\
 \frac{-\alpha - 1}{2} & \frac{\alpha + 1}{2} & \frac{\alpha + 1}{4} & \frac{-\alpha - 1}{4}  \\
 \frac{\alpha + 1}{2} & \frac{-\alpha - 1}{2}  & \frac{-\alpha - 1}{4} & \frac{\alpha+1}{4} \\
\end{array}
\right)\]

\newpage

\subsection{Number of leaves of each colour}\label{App:leaves}

For $p< (3 - \alpha)/4$:

\[ \Sigma_I^l = \frac{\alpha + 1}{4(2+\alpha)^2(3+\alpha)(2p-3)(4p+\alpha-3)}
\left(
\begin{array}{cccc}
\sigma_{1,1}^l & \sigma_{1,2}^l & \sigma_{1,3}^l & \sigma_{1,4}^l\\
\sigma_{2,1}^l & \sigma_{2,2}^l & \sigma_{2,3}^l & \sigma_{2,4}^l \\
\sigma_{3,1}^l & \sigma_{3,2}^l & \sigma_{3,3}^l & \sigma_{3,4}^l \\
\sigma_{4,1}^l & \sigma_{4,2}^l & \sigma_{4,3}^l & \sigma_{4,4}^l \\
\end{array}
\right)\]
where 
\begin{align*}
\sigma_{1,1}^c = \sigma_{2,2}^c &=    \left(8 p^2-6 p-1\right) \alpha^3+\left(48 p^2-46 p+1\right) \alpha^2\\
&\hspace{20mm}+\left(112 p^2-158 p+49\right) \alpha+88 p^2-158 p+71 \\
\sigma_{2,1}^c = \sigma_{1,2}^c &= -\left(\left(8 p^2-6 p-1\right) \alpha^3+\left(48 p^2-50p+7\right) \alpha^2 \right.\\
&\hspace{20mm}\left. +\left(96 p^2-126 p+37\right) \alpha+72 p^2-122p+53\right)  \\
\sigma_{3,3}^c = \sigma_{4,4}^c &= -(\alpha+1) \left((2 p-3) \alpha^4+2 \left(4p^2-7\right) \alpha^3+4 \left(10 p^2-9 p-4\right) \alpha^2\right.\\
&\hspace{20mm}\left.+\left(56 p^2-60 p-4\right) \alpha+8 p^2+14 p-23\right)  \\
\sigma_{3,1}^c = \sigma_{4,2}^c &=\sigma_{1,3}^c = \sigma_{2,4}^c =(\alpha+1) \left((2 p-1) \alpha^3+\left(-8 p^2+22 p-9\right) \alpha^2\right.\\
&\hspace{20mm}\left. +\left(-32 p^2+62 p-21\right) \alpha-40 p^2+74 p-29\right) \\
\sigma_{3,2}^c = \sigma_{4,1}^c &= -(\alpha+1) \left((2 p-1) \alpha^3+\left(-8 p^2+22 p-9\right) \alpha^2\right.\\
&\hspace{20mm} \left.+\left(-32 p^2+66 p-27\right) \alpha-24 p^2+38 p-11\right)   \\
\sigma_{3,4}^c = \sigma_{4,3}^c &=(\alpha+1)^2 \left((2p-3) \alpha^3+\left(8 p^2-2 p-11\right) \alpha^2\right.\\
&\hspace{20mm} \left.+\left(32 p^2-34 p-5\right) \alpha+24 p^2-22 p-5\right)
\end{align*}

\noindent For $p = (3-\alpha)/4$:

\[\Sigma_{II}^l = \left(
\begin{array}{cccc}
\frac{(\alpha - 1)^2(\alpha+1)}{4(3+\alpha)^2} & -\frac{(\alpha - 1)^2(\alpha+1)}{4(3+\alpha)^2} & 
-\frac{(\alpha + 1)^2(\alpha-1)}{2(3+\alpha)^2} & \frac{(\alpha + 1)^2(\alpha-1)}{2(3+\alpha)^2}  \\
-\frac{(\alpha - 1)^2(\alpha+1)}{4(3+\alpha)^2} & \frac{(\alpha - 1)^2(\alpha+1)}{4(3+\alpha)^2} & 
\frac{(\alpha + 1)^2(\alpha-1)}{2(3+\alpha)^2} & -\frac{(\alpha + 1)^2(\alpha-1)}{2(3+\alpha)^2}  \\
-\frac{(\alpha - 1)^2(\alpha+1)}{4(3+\alpha)^2} & \frac{(\alpha - 1)^2(\alpha+1)}{4(3+\alpha)^2} & 
\frac{(1+\alpha)^3}{(3+\alpha)^2} & -\frac{(1+\alpha)^3}{(3+\alpha)^2} \\
\frac{(\alpha - 1)^2(\alpha+1)}{4(3+\alpha)^2} & -\frac{(\alpha - 1)^2(\alpha+1)}{4(3+\alpha)^2} & 
-\frac{(1+\alpha)^3}{(3+\alpha)^2} & \frac{(1+\alpha)^3}{(3+\alpha)^2} 
\end{array}
\right)\]

\noindent For $p=1/2$:
\[ \Sigma_I^l = \frac{\alpha + 1}{4(2+\alpha^2)(3 + \alpha)}\left(
\begin{array}{ccc}
7 + 6\alpha + \alpha^2 & -5-4\alpha - \alpha^2 & -2(1 + \alpha) \\
-5-4\alpha - \alpha^2 & 7 + 6\alpha + \alpha^2 & -2(1 + \alpha) \\
-2(1 + \alpha) & -2(1 + \alpha) & 4(1 + \alpha)
\end{array}\right)\]

\section{Proofs of Lemmas \ref{lem:superb} and \ref{lem:crit}}\label{App:LemmaProofs}

\begin{proof}[Proof of Lemma \ref{lem:superb}]
Using the recursion in \eqref{eq:recursion}, where $\phi(\delta) = d!/(d-\delta)!$ (recall \eqref{eq:phi}), we get the following partial differential equation:
\begin{align}
\frac{\partial}{\partial x}R(x,u) &= \sum_{n,k} n \frac{r_{n,k}}{n!}x^{n-1}u^k \nonumber\\
&= u\sum_{n,k} \sum_{\delta=0}^{d}\sum_{s=0}^\delta \binom{\delta}{s}\frac{\phi(\delta)}{\delta!}\sum_{n_1, \ldots , n_\delta}\sum_{k_1, \ldots, k_s}\prod_{i=1}^s \frac{pr_{n_i,k_i}x^{n_i}u^{k_i}}{n_i!} \prod_{j=s+1}^\delta \frac{(1-p)b_{n_j}x^{n_j}}{n_j!} \nonumber \\
&= u \sum_{\delta=0}^{d} \binom{d}{\delta} \sum_{s=0}^\delta \binom{\delta}{s} (pR(x,u))^s((1-p)B(x))^{\delta-s} \nonumber \\
&= u \sum_{\delta=0}^{d} \binom{d}{\delta}\left(pR(x,u) + (1-p)B(x)\right)^\delta \nonumber \\
&= u\left(1 + pR(x,u) + (1-p)B(x)\right)^{d} \nonumber \\
&= u\left(1 +pR(x,u) + (1-p)\left((1 - (d-1)x)^{-\frac{1}{d-1}} - 1\right)\right)^{d}. \label{eq:baryRgen}
\end{align}

We proceed by strong induction. Using the above differential equation, we see that 
\begin{equation}\label{eq:baryR1}
R_1'(x) =\left.\frac{\partial^2}{\partial u \partial x}R(x,u)\right|_{u = 1} = \frac{pd}{(1-(d-1)x)}R_1(x) + ((1-(d-1)x)^{\frac{-d}{d-1}}.
\end{equation}
Solving this differential equation with the initial condition $R_1(0) = 0$ yields
\begin{equation*}
R_1(x) = \frac{1}{pd-1}\left((1 - (d-1)x)^{\frac{-pd}{d-1} }- (1 - (d-1)x)^{-\frac{1}{d-1}}\right),
\end{equation*}
which is analytic on the desired cut plane.

For the inductive step, using the product rule at higher orders of partial differentiation produces 
\begin{align}
R_k'(x) &= \left.\frac{\partial^{k+1}}{\partial u^k \partial x}R(x,u)\right|_{u = 1} \nonumber\\
&= \left.\frac{\partial^k}{\partial u^k}u\left(1 +pR(x,u) + (1-p)\left((1 - (d-1)x)^{-\frac{1}{d-1}} - 1\right)\right)^{d}\right|_{u=1} \label{eq:Rk}\\
&=\left. \left( u \frac{\partial^k}{\partial u^k}f(R(x,u)) + k \frac{\partial^{k-1}}{\partial u^{k-1}}f(R(x,u))\right)\right|_{u=1}\nonumber
\end{align}
where $f(y) = \left(1 + py + (1-p)\left((1-(d-1)x)^{-\frac{1}{d-1}} - 1\right)\right)^d.$
Define
\[ f^{(m)}(y) := \frac{d^m}{dy^m}f(y).\]
Then 
\[f^{(m)}(y) = \frac{d!}{(d-m)!}p^m\left(1 + py + (1-p)\left((1-(d-1)x)^{-\frac{1}{d-1}} - 1\right)\right)^{d-m}\]
for $0 \leq m \leq d$, and $f^{(m)} = 0$ for $m > d$. In particular
\begin{equation}\label{eq:newf(R0)}
 f^{(m)}(R_0(x)) =  \frac{d!}{(d-m)!}p^m(1-(d-1)x)^{-\frac{d-m}{d-1}}
\end{equation}
for $0 \leq m \leq d$. 
By using Fa\'a di Bruno's formula for higher order chain rule, we see that 
\begin{align}
\left.\frac{\partial^k}{\partial u^k}f(R(x,u))\right|_{u=1} &= \sum_{j=1}^k f^{(j)}(R_0(x))B_{k,j}(R_1(x), \ldots, R_{k-j+1}(x)) \nonumber\\
&= pd(1-(d-1)x))^{-1}R_k(x) + h_k(x), \label{eq:fnew}
\end{align}
where 
\[ h_k(x) = \sum_{j=2}^k f^{(j)}(R_0(x))B_{k,j}(R_1(x), \ldots, R_{k-j+1}(x)).\]
Since analyticity is preserved under arithmetic operations as well as integration, the analyticity of $R_k(x)$ on the desired cut plane follows by the induction hypothesis. By using the forms of $R_j(x)$ in the induction hypothesis and \eqref{eq:newf(R0)}, then 
\begin{align*}
 h_k(x) &=  H_k(1-(d-1)x)^{\frac{-kpd +k - d}{d-1}} + O\left((1-(d-1)x)^{\frac{-kpd +k - d}{d-1}+ \varepsilon}    \right),
\end{align*}
where 
\[ H_k = \sum_{j=2}^{\min\{k,d\}} \frac{p^jd!}{(d-j)!}B_{k,j}(D_1, \ldots, D_{k-j+1}).\]
From \eqref{eq:fnew}, the induction hypothesis, and the assumption $p > 1/d$, we can also conclude that,
\begin{align*}
 \left.k\frac{\partial^{k-1}}{\partial u^{k-1}}f(R(x,u))\right|_{u=1} &= O\left((1 - (d-1)x))^{\frac{-kpd +pd +k - 1-d}{d-1}} \right)\\
& =O\left((1-(d-1)x)^{\frac{-kpd +k - d}{d-1}+ \varepsilon}    \right)
\end{align*}
for some $\varepsilon > 0$. By solving the differential equation
\begin{align*}
R_k'(x) &= \left. \left( u \frac{\partial^k}{\partial u^k}f(R(x,u)) + k \frac{\partial^{k-1}}{\partial u^{k-1}}f(R(x,u))\right)\right|_{u=1} \\
&=  pd(1-(d-1)x))^{-1}R_k(x)+ H_k(1-(d-1)x)^{\frac{-kpd +k - d}{d-1}}\\
&\hspace{20mm} + O\left((1-(d-1)x)^{\frac{-kpd +k - d}{d-1}+ \varepsilon} \right), 
\end{align*}
we get that 
\[ R_k(x) = \frac{H_k}{(k-1)(pd-1)}(1 - (d-1)x)^{\frac{-kpd + k -1}{d-1}} + O\left((1 - (d-1)x)^{\frac{-kpd + k -1}{d-1} + \varepsilon}\right). \]
Setting 
\[ D_k = \frac{H_k}{(k-1)(pd-1)}\]
concludes the proof of the lemma. 
\end{proof}

\begin{proof}[Proof of Lemma \ref{lem:crit}]
The derivation in \eqref{eq:baryRgen} applies here as well. Solving the differential equation in \eqref{eq:baryR1} with $p=1/d$ yields
\[ R_1(x) = \frac{-1}{d-1}(1-(d-1)x)^{\frac{-1}{d-1}}\ln(1-(d-1)x),\]
which is analytic on the desired cut plane. 

For the inductive step, the derivation \eqref{eq:Rk} holds here as well. We get that 
\[ f^{(m)}(R_0(x)) = \frac{d!}{d^m(d-m)!}(1-(d-1)x)^{-\frac{d-m}{d-1}}.\]
By following the same steps as the proof of Lemma \ref{lem:superb}, we see that 
\begin{align*}
\left.\frac{\partial^k}{\partial u^k}f(R(x,u))\right|_{u=1} &= \sum_{j=1}^k f^{(j)}(R_0(x))B_{k,j}(R_1(x), \ldots, R_{k-j+1}(x)) \\
&= (1-(d-1)x)^{-1}R_k(x) + l_k(x),
\end{align*}
where 
\[l_k(x)= \sum_{j=2}^k f^{(j)}(R_0(x))B_{k,j}(R_1(x), \ldots, R_{k-j+1}(x)).\]
As before, analyticity is preserved. By using the induction hypothesis and the simplification 
\[ B_{k,2}(x_1, \ldots, x_{k-1}) = \frac{1}{2}\sum_{i=1}^{k-1}\binom{k}{i}x_ix_{k-i},\] then 
\begin{align*}
 l_k(x) &= L_k(1-(d-1)x)^{\frac{-d}{d-1}}\ln^{2k-2}(1-(d-1)x) \\
&\hspace{20mm} + O((1-(d-1)x)^{\frac{-d}{d-1}}\ln^{2k-3}(1-(d-1)x)),
\end{align*}
where 
\[ L_k =\frac{ d-1}{2d}\sum_{j=1}^{k-1}\binom{k}{j}E_jE_{k-j}.\]
We can also conclude that 
\[  \left.k\frac{\partial^{k-1}}{\partial u^{k-1}}f(R(x,u))\right|_{u=1} = O\left((1-(d-1)x)^{\frac{-d}{d-1}}\ln^{2k-3}(1-(d-1)x)\right).\]
Solving the differential equation 
\begin{align*}
 R_k'(x) &= (1-(d-1)x)^{-1}R_k(x) + L_k(1-(d-1)x)^{\frac{-d}{d-1}}\ln^{2k-2}(1-(d-1)x) \\
&\hspace{20mm} + O\left((1-(d-1)x)^{\frac{-d}{d-1}}\ln^{2k-3}(1-(d-1)x)\right),
\end{align*}
we get that 
\begin{align*}
 R_k(x) &= \frac{-L_k}{(2k-1)(d-1)}(1-(d-1)x)^{\frac{-1}{d-1}}\ln^{2k-1}(1-(d-1)x) \\
&\hspace{20mm} + O\left((1-(d-1)x)^{\frac{-1}{d-1}}\ln^{2k-2}(1-(d-1)x)\right).
\end{align*}
Setting 
\[ E_k = \frac{L_k}{(2k-1)(d-1)}\]
concludes the proof of the lemma. 
\end{proof}



\begin{thebibliography}{abcd:99}

\bibitem{ADLV:20} L. Addario-Berry, L. Devroye, G. Lugosi, and V. Velona, Broadcasting on random recursive trees, preprint available online at {\tt arXiv:2006.11787} (accessed January 2021).

\bibitem{BAAL:99} A. L. Barab\'asi and R. Albert, Emergence of scaling in random networks, {\em Science} {\bf 286.5439} (1999), 509--512.

\bibitem{BAUR:20} E. Baur, On a class of random walks with reinforced memory, {\em J. Stat. Phys.} {\bf 181} (2020), 772--802. 

\bibitem{BABE:15} E. Baur and J. Bertoin, The fragmentation process of an infinite recursive tree and Ornstein-Uhlenbeck type processes, {\em Electron. J. Probab.} {\bf 20} (2015), 1--20. 

\bibitem{BABE:16} E. Baur and J. Bertoin, Elephant random walks and their connecton to P\'olya-type urns, {\em Phys. Rev. E} {\bf 94} (2016), 052134.

\bibitem{BEFS:92} F. Bergeron, P. Flajolet, and B. Salvy, Varieties of increasing trees, in {\em Colloquium on trees in algebra and programming}, 24--48, Springer, 1992.

\bibitem{BRTS:01} B. Bollob\'as, O. Riordan, J. Spencer, and G. Tusn\'ady, The degree sequence of a scale-free random graph process, {\em Random Structures Algorithms} {\bf 18} (2001), 279--290.

\bibitem{BUSI:18} S. Businger, The shark random swim, {\em J. Stat. Phys. } {\bf 172} (2018), 701--717. 

\bibitem{COMT:74} L. Comtet, {\em Advanced combinatorics. The art of finite and infinite expansions}, D.~Reidel Publishing Co.,  1974.

\bibitem{DRMO:09} M. Drmota, {\em Random trees: an interplay between combinatorics and probability}, Springer Science and Business Media, 2009.

\bibitem{DWAS:69} M. Dwass, The total progeny in a branching process and a related random walk. {\em J. Appl. Probab.} {\bf 6} (1969), 682--685.

\bibitem{EKPS:00} W. Evans, C. Kenyon, Y. Peres, and L. J. Schulman, Broadcasting on trees and the Ising model, {\em Ann. Appl. Probab.} {\bf 10} (2000), 410--433.

\bibitem{FLSE:09} P. Flajolet and R. Sedgewick, {\em Analytic Combinatorics}, Cambridge University Press, 2009.

\bibitem{GULM:18} L. Gulikers, M. Lelarge, and L. Massouli\'e, An impossibility result for reconstruction in the degree-corrected stochastic block model, {\em Ann. Appl. Probab.} {\bf 28} (2018), 3002--3027. 

\bibitem{GUT:13} A. Gut, {\em Probability: A Graduate Course} (2nd ed.), Springer Texts in Statistics, 2013.

\bibitem{HEHO:17} M. Heydenreich and R. Van der Hofstad, {\em Progress in high-dimensional percolation and random graphs}, Springer, 2017. 

\bibitem{HOJS:17} C. Holmgren, S. Janson, and M. \v{S}ileikis, Multivariate normal limit laws for the numbers of fringe subtrees in $m$-ary search trees and preferential attachment trees, {\em Electron. J. Combin.} {\bf 24} (2017), Paper 2.51, 49pp.

\bibitem{JANS:04} S. Janson, Functional limit theorems for multitype branching processes and generalized P\'olya urns, {\em Stochastic Process. Appl.} {\bf 110} (2004), 177--245. 

\bibitem{JANS:05} S. Janson, Asymptotic degree distributions in random recursive trees, {\em Random Structures Algorithms} {\bf 26} (2005), 69--83. 

\bibitem{JUNG:31} R. Jungen, Sur les séries de Taylor n'ayant que des singularités algébrico-logarithmiques sur leur cercle de convergence, {\em Comment. Math. Helv.} {\bf 3} (1931), 266--306.

\bibitem{KURS:16} R. Kürsten, Random recursive trees and the elephant random walk, {\em Phys. Rev. E} {\bf 93} (2016), 032111. 

\bibitem{LYPE:16} R. Lyons and Y. Peres, {\em Probability on Trees and Networks}, Cambridge University Press, 2016. 

\bibitem{MASM:92} H. M. Mahmoud and R. T. Smythe, Asymptotic joint normality for outdegrees of nodes in random recursive trees, {\em Random Structures Algorithms} {\bf 3} (1992), 255--266. 

\bibitem{MILE:05} G. Mittag-Leffler, Sur la représentation analytique d'une branche uniforme d'une fonction monogène (cinquième note), {\em Acta Math. } {\bf 29} (1905), 101--181. 

\bibitem{MOHL:15} M. Möhle, The Mittag-Leffler process and a scaling limit for the block counting process of the Bolthausen-Sznitman coalescent, {\em ALEA Lat. Am. J. Probab. Math. Stat.} {\bf 12} (2015), 35--53. 

\bibitem{MOSS:04} E. Mossel, Survey - information flow on trees, {\em DIMACS Ser. Discrete Math. Theoret. Comput. Sci.} {\bf 63} (2004), 155--170. 

\bibitem{MONS:16} E. Mossel, J. Neeman, and A. Sky, Belief propagation, robust reconstruction and optimal recovery of block models, {\em Ann. Appl. Probab.} {\bf 26} (2016), 2211--2256.

\bibitem{NAHE:82} D. Najock and C. C. Heyde, On the number of terminal vertices in certain random trees with an application to stemma construction in philology, {\em J. Appl. Probab.} {\bf 19} (1982), 675--680. 

\bibitem{OTTE:49} R. Otter, The multiplicative process, {\em Ann. Math. Statistics} {\bf 20} (1949), 206--224. 

\bibitem{PITT:94} B. Pittel, Note on the heights of random recursive trees and random $m$-ary trees, {\em Random Structures Algorithms} {\bf 5} (1994), 337--347.

\bibitem{SZYM:87} J. Szyma\'nski, On a nonuniform random recursive tree, in {\em North-Holland Mathematics Studies} (Vol. 144), 297--306, Elsevier, 1987. 

\bibitem{TAMY:67} M. Tapia and B. Myers, Generation of concave node-weighted trees, {\em IEEE Trans. Circuit Theory} {\bf 14} (1967), 229--230.


\end{thebibliography}
\end{document}